 \documentclass[letterpaper,11 pt]{article}
 
\usepackage{amsmath} 
\usepackage{amssymb}  
\usepackage{acronym,wrapfig,plain,mathrsfs,enumerate,relsize,color}
\usepackage[colorlinks=true,citecolor=blue]{hyperref}
\definecolor{blue}{RGB}{0,0,255} 

\providecommand{\keywords}[1]{\textbf{\textit{keywords:}} #1}
\allowdisplaybreaks
\usepackage{graphicx}
\usepackage{spverbatim}
\newtheorem{algorithm}{Algorithm}
\usepackage{algorithm}
\usepackage{amsmath}
\usepackage{multirow}
\usepackage{algorithmic}
\usepackage{subfig}
\usepackage{hyperref}       
\usepackage{url}            
\usepackage{booktabs}       
\usepackage{amsfonts}       
\usepackage{nicefrac}       
\usepackage{microtype}      
\usepackage{xcolor}         
\usepackage{graphicx}
\usepackage{appendix}
 
\usepackage{amsfonts,amsmath,fullpage,bbm}
\usepackage{amssymb,amsthm,multirow,verbatim}
\usepackage{acronym,wrapfig,plain,mathrsfs,enumerate,relsize,color}

\usepackage{amsmath,amsfonts}
\usepackage{algorithmic}
\usepackage{algorithm}
\usepackage{array}
\usepackage{witharrows}
\providecommand{\keywords}[1]{\textbf{\textit{keywords:}} #1}
\allowdisplaybreaks
\usepackage{spverbatim}
\usepackage{mdframed}
\usepackage{algorithm}
\usepackage{multirow}
\usepackage{tcolorbox}
\newtheorem{theorem}{Theorem}
\newtheorem{corollary}{Corollary}
\newtheorem{lemma}{Lemma}

\newtheorem{remark}{Remark} 
\newcommand{\argmin}{\mathop{\rm argmin}}

\newtheorem{assumption}{Assumption}
\newtheorem{definition}{Definition}

\newcommand{\za}[1]{{\color{black}#1}}

\newcommand{\ze}[1]{{\color{black}#1}}
\newcommand{\zal}[1]{{\color{black}#1}}
\newcommand{\z}[1]{{\color{black}#1}}
\newcommand{\zz}[1]{{\color{black}#1}}
\newcommand{\zr}[1]{{\color{black}#1}}

\newcommand{\aj}[1]{{\color{black}#1}}

\title{\LARGE \bf
Convergence Analysis of Non-Strongly-Monotone Stochastic Quasi-Variational Inequalities
}

\author{
  \begin{tabular}{c}
    Zeinab Alizadeh \textsuperscript{*} \qquad Afrooz Jalilzadeh\footnote{Department of Systems and Industrial Engineering, The University of Arizona, Tucson, AZ, USA.\\
     \texttt{\{zalizadeh, afrooz\}@arizona.edu}}
  \end{tabular}
}
\date{}

\onecolumn
\begin{document}

\maketitle
\thispagestyle{empty}
\pagestyle{empty}


\begin{abstract}
While Variational Inequality (VI) is a well-established mathematical framework that subsumes Nash equilibrium and saddle-point problems, less is known about its extension, Quasi-Variational Inequalities (QVI). QVI allows for cases where the constraint set changes as the decision variable varies allowing for a more versatile setting. In this paper, we propose extra-gradient and gradient-based methods for solving a class of monotone Stochastic Quasi-Variational Inequalities (SQVI) and establish a rigorous convergence rate analysis for these methods. Our approach not only advances the theoretical understanding of SQVI but also demonstrates its practical applicability. Specifically, we highlight its effectiveness in reformulating and solving problems such as generalized Nash Equilibrium, bilevel optimization, and saddle-point problems with coupling constraints.
\\

\end{abstract}

\section{Introduction}
Variational Inequality (VI) problems find applications in various areas like convex Nash games, traffic equilibrium, and economic equilibrium \cite{facchinei2007finite}.
The Stochastic Variational Inequality (SVI) extends VI theory to address decision-making problems involving uncertainty \zr{\cite{shanbhag2013stochastic,gurkan1996sample,koshal2012regularized,iusem2017extragradient,alizadeh2024randomized}}.
An extension of VIs, known as Quasi-Variational Inequality (QVI), emerges when the constraint set depends on the decision variable. From a Nash equilibrium game perspective, QVI captures the interdependencies among players' strategy sets, particularly in scenarios involving shared resources.

In this paper, we consider a Stochastic QVI (SQVI) problem. In particular, the goal is to find $x^*\in K(x^*)$ such that
\begin{align*}\label{prob:SQVI}\tag{\bf SQVI}
\langle F(x^*), y-x^* \rangle\geq 0, \quad \forall y\in K(x^*),
\end{align*}
where $K:X\to 2^X$ is a lower-hemicontinous set-valued mapping with non-empty, closed and convex values \aj{such that} $K(x)\subseteq X$ for all $x\in X$, $X\subseteq \mathbb R^n$ is a convex and compact set,  $F(x)\triangleq \mathbb E[G(x,\xi)]$, $\xi: \Omega \to \mathbb R^d$, ${G}: X \times \mathbb R^d  \rightarrow
\mathbb{R}^n$, and the associated probability space is denoted by $(\Omega, {\cal F}, \mathbb{P})$.

Although the theoretical results and algorithm development for VIs are rich and fruitful \cite{alizadeh2024randomized,jalilzadeh2019proximal,korpelevich1976extragradient,doi:10.1137/S1052623403425629,tseng2000modified,nesterov2007dual,malitsky2015projected,malitsky2020forward}, research studies on QVIs remain limited and most of the existing methods for solving VIs are not amendable for solving \eqref{prob:SQVI} which calls for the development of new techniques and iterative methods. In particular, the primary focus of existing research studies for QVIs is on solution existence \cite{ravat2017existence} and the development of algorithms often requires restrictive assumptions such as strong monotonicity \cite{mijajlovic2019gradient}. Moreover, some efforts have been made to develop algorithms for solving generalized Nash game and saddle point (SP) problems with coupled constraints; however, their convergence typically requires the constraints to
be linear \cite{dai2024optimality,jordan2023first} or jointly convex \cite{alizadeh2024randomized,boob2023first}. To the best of our knowledge, there is no method
with convergence guarantees for general nonlinear constraints. To fill this gap, in this paper, we aim to develop efficient inexact iterative methods for solving \eqref{prob:SQVI} under less restrictive assumptions with convergence rate guarantees. 
In the deterministic setting, several studies have explored numerical approaches for solving QVIs \cite{pang2005quasi,antipin2013second,facchinei2014solving,mijajlovic2015proximal,noor2000new,ryazantseva2007first,salahuddin2004projection,noor2007existence}. Notably, \cite{mijajlovic2019gradient} and \cite{nesterov2006solving} demonstrated a linear convergence rate for the strongly monotone QVI problem. In the stochastic regime, in our previous work \cite{alizadeh2022inexact}, we obtained a linear convergence rate under strong monotonicity assumptions. In this paper, we extend the result to monotone SQVIs, where the operator $F$ satisfies the quadratic growth property (see Definition \ref{Quadratic growth def}).

\subsection{Outline of the Paper} \aj{In the next section, we explore how the SQVI problem fits into well-known problems in optimization and game theory. In Section \ref{contribution}, we outline the key contributions of our work. Then, in Section \ref{assum}, we lay out the key assumptions needed for our convergence analysis. Section \ref{prop} introduces inexact extra-gradient and gradient-based algorithms, and in Section \ref{numeric}, we show the effectiveness of the proposed methods in solving real-world examples.
In Section \ref{summary}, we summarize our key findings and propose directions for future research.}

\section{Applications}\label{sec:applications}
\subsection{Generalized Nash Equilibruim}
\aj{Nash equilibrium (NE) is a fundamental concept in game theory where a finite collection of selfish agents compete with each other and seek to optimize their own individual objectives. An NE is described as a collection of specific strategies chosen by all the players, where no player can reduce their cost by unilaterally changing their strategy within their feasible strategy set. The Generalized Nash equilibrium Problem (GNEP) is an extension of the classic NE Problem (NEP), in which   players interact
also at the level of the feasible sets, i.e., each player's strategy set depends on other
players’ strategies. This situation arises naturally if the players share some common resource, e.g.,
a communication link, an electrical transmission line, etc. GNEP has been extensively utilized in the literature to formulate applications arising in economics, operations research, and other fields \cite {facchinei2010generalized,krilavsevic2023learning}.
Consider $N$ players each with cost function $f_i(x_i,x_{(-i)})\triangleq \mathbb E[h(x_i,x_{(-i)},\xi)]$ for $i=1\hdots,N$, where $x_i$ is the strategy of player $i$ , $x_{(-i)}$ is the strategy of other players, { and $\xi: \Omega \to \mathbb{R}^d$ is a random variable defined on the probability space $(\Omega, \mathcal{F}, \mathbb{P})$}. Each player $i$'s objective is to solve the following optimization problem:} 
\begin{align*}
    \min_{x_i} f_i(x_i,x_{(-i)}) \quad \hbox{s.t.}\quad x_i \in K_i(x_{(-i)}),
\end{align*}
where $K_i(x_{(-i)}) =\{ x_i  \in \mathbb R^{n_i} |  g_i(x_i,x_{(-i)})\leq 0\}  $ is a closed convex set-valued map.
$f_i,g_i:\mathbb R^n\times\mathbb R^m\to \mathbb R$ are continuously differentiable. By defining $K(x)=\prod_{i=1}^N K_i(x_{(-i)})$ and $F(x) =[ \nabla_{x_i} f_i(x)]_{i=1}^N $, finding a GNE will be equivalent to solve the following SQVI problem\zr{:}
$\langle F(x^*), x-x^* \rangle \geq 0, \forall x \in K(x^*)$.

\subsection{Bilevel Optimization} 
A bilevel optimization problem is hierarchical decision-making where one optimization problem is nested within another. This problem has a variety of applications, including data hyper cleaning, hyper-parameter optimization, and meta-learning in machine learning \cite{li2020improved,franceschi2018bilevel}. This problem can be formulated as follows:
\begin{align}\label{eq:bilevel}
    \min_{u\in U}~f(u,w(u))\quad \hbox{s.t.}\quad w(u)\in\argmin_{w\in W}~g(u,w),
\end{align}
where $f(u,w)\triangleq\mathbb E[h(u,w,\xi)]$, $f,g:\mathbb R^n\times\mathbb R^m\to \mathbb R$ are continuously differentiable, $U\subseteq \mathbb R^n$ and $W\subseteq \mathbb R^m$ are convex and compact sets. This problem can be equivalently represented as the following minimization problem: 
\begin{align}\label{eq:bilevel-single}
    \min_{u\in U,w\in W}~f(u,w)\quad \hbox{s.t.}\quad w\in S(u),
\end{align}
where $S(u)$ is the solution to the lower-level problem in \eqref{eq:bilevel} for any given $u\in U$. Note that \eqref{eq:bilevel-single} in its general form can be NP-hard, therefore, it is imperative to seek a first-order stationary solution, i.e., finding $(u^*,w^*)$ such that $\langle \nabla_u f(u^*,w^*), u-u^* \rangle+\langle \nabla_w f(u^*,w^*), w-w^* \rangle \geq 0$ for any $u\in U$ and $w\in S(u^*)$, which is a special case of \eqref{prob:SQVI} by defining $x\triangleq [u^\top,w^\top]^\top$, $F(x)\triangleq [\nabla_u f(u,w),\nabla_w f(u,w)]^\top$, and $K(x)\triangleq U\times S(u)$. 

\subsection{Saddle Point Problems with Coupling Constraints}
\aj{Recently, minimax optimization problem, i.e., $\min_{u\in U}\max_{w\in W} f(u,w)$, has received considerable attention due to its relevance to a wide range of machine learning (ML) problems such as reinforcement learning, generative adversarial networks (GAN), fairness in machine learning, and generative adversarial imitation learning. Convex-concave minimax problems can be also viewed from a game theory perspective in which the objective function severs as a payoff function and one player aims to minimize the payoff while the second player's goal is to maximize it. In this context, a saddle point solution $(u^*,w^*)$ serves as both a minimum of $f$ in the $u$-direction and a maximum of $f$ in the $w$-direction. 
Here, $w^*$ represents the inner player's best response to their opponent's strategy $u^*$, and a saddle point $(u^*,w^*)$ is also called a Nash equilibrium (NE).

Here, we consider a more general SP problem where the constraint depends on the decisions of both players, i.e.,
\begin{align}\label{minimax}
    \min_{u\in U}\max_{w\in W} f(u,w) \quad \hbox{s.t.}\quad  g(u,w)\leq 0,
\end{align} 
where $f(u,w)\triangleq\mathbb E[h(u,w,\xi)]$, $U$ and $W$ are convex sets. Such problems have numerous applications in various fields such as adversarial attacks in network flow problems \cite{tsaknakis2023minimax}. Because of the dependency of the constraint on both variables, if $g$ is not jointly convex in both $x$ and $y$ then this problem cannot be formulated as traditional VI but we can reformulate it as QVI. From the first order optimality condition of \eqref{minimax}, we have that 
$$\langle \nabla_x f(u^*,w^*), u-u^* \rangle \geq 0, \forall u\in \{u\in U\mid g(u,w^*)\leq 0\}  $$
$$\langle \nabla_w f(u^*,w^*), w^*-w \rangle \geq 0, \forall w\in \{ w \in W | g(u^*,w) \leq 0 \}.  $$
Defining $F(x^*) = \begin{bmatrix} \nabla_u f(x^*) \\ -\nabla_\z{w} f(x^*) \end{bmatrix}$ and $K(x^*)\triangleq U(w^*) \times W(u^*)$, solving \eqref{minimax} will be equivalent to solving the following SQVI problem.
$\langle F(x^*), x-x^* \rangle \geq 0, \forall x \in K(x^*)$.}

\section{Contributions}\label{contribution}
\aj{Motivated by the absence of efficient methods for solving non-strongly monotone SQVI problems, we propose extra-gradient and gradient-based schemes for a class of monotone SQVI where the operator satisfies a quadratic growth property. Non-strongly monotone problems are particularly challenging due to the potential lack of a unique solution, which complicates the convergence analysis. Moreover, in QVIs, the constraint set varies with the decision variable, making it even more difficult to establish convergence guarantees.
In this work, we address these challenges by developing methods that not only handle the complexity of the monotone setting but also manage inexact projections when direct computation of the projection onto the constraint set is difficult. Using an inexact approach, we establish the first known convergence rate result for this class of problems, achieving linear convergence in terms of outer iterations. Furthermore, we show an oracle complexity of $\mathcal O(1/\epsilon^2)$; that is, to achieve an $\epsilon$-solution, $\mathcal O(1/\epsilon^2)$ sample operators need to be computed. The oracle complexity will improve to $\mathcal O(\log(1/\epsilon))$ for the deterministic case. To the
best of our knowledge, our proposed algorithms are the first
with a convergence rate guarantee to solve non-strongly monotone SQVIs.}

\section{Preliminaries}\label{assum}
\aj{In this section,  first, we define important notations and
then we present the definitions, assumptions, and essential technical lemmas required for our convergence analysis.}
\subsection{Notations}
 Throughout the paper, $\|x\|$ denotes the Euclidean vector norm, i.e., $\|x\|=\sqrt{x^Tx}$.  \aj{$\mathcal{P}_K [x]$} is the projection of $x$ onto the set $K$, i.e. \aj{$\mathcal{P}_K [x] = \mbox{argmin}_{y \in K} \| y-x \|$}. $\mathbb E[x]$ is used to denote the expectation of a random variable $x$. \aj{We let $X^*$ denote the  set of optimal solution of \eqref{prob:SQVI} problem, which is assumed to be nonempty.}
\subsection{Quadratic Growth Property}\label{sec:QG}
\aj{In this paper, we consider a monotone operator $F$ that has a quadratic growth property, which is defined next. } 
\begin{definition}\label{Quadratic growth def}
\aj{An operator $F$ has a quadratic growth (QG) property
on set $X$ if there exists a constant $\z{\mu_F} > 0 $ such that for any $x \in X $ and some $y \in \mathcal{P}_{X^*}[x]$ 
\begin{align*}
    \langle  F(x)- F(y),x-y \rangle\geq \z{\mu_F}  \|x-y\|^2, \quad \forall x\in X.
\end{align*}}
\end{definition}
\aj{It is worth noting that, unlike the strong monotonicity assumption, the QG property does not imply a unique solution. In fact, QG property is a weaker assumption than strong monotonicity \cite{bello2022quadratic,necoara2019linear}.
As an example of a QG operator, consider function $f(x)\triangleq g(Ax)+c^Tx$, where $g$ is a smooth and strongly convex function, $A\in \mathbb R^{n\times m}$ is a nonzero general matrix and $c\in \mathbb R^n$. One can show that $\nabla f(x)$ satisfies the QG property \cite{necoara2019linear} while $\nabla f(x)$ may not be strongly monotone unless $A$ has a full column rank. See \cite{bello2022quadratic} and \cite{necoara2019linear} for more examples.}

Next, we consider a class of generalized Nash equilibrium problems that can be reformulated as QVIs.
For this class, we explicitly characterize the equilibrium set $X^*$ and show that the associated QVI operator satisfies the QG property.

{\bf Class of QVIs with Quadratic Growth.} Consider a generalized Nash game in which each player $i\in\{1,\cdots,N\}$ solves the following optimization problem:
\begin{align*}
    &\min_{x_i\in X_i} f_i(A_ix_i)+c_i^Tx^*_{(-i)}\\
    &\mbox{s.t.}\quad B_ix_i+C_ix_{(-i)}^*=d_i,
\end{align*}
where $X_i\triangleq\{x_i\mid D_ix_i\leq q_i\}$ is a polyhedral set, function $f_i$ is strongly convex, $x^*_{(-i)}$ denotes the equilibrium decisions of the other players. Note that the problem may not be strongly monotone, when $A$ is not full-rank. This game can be equivalently reformulated as a quasi variational inequality (QVI) problem. 

Let $x = (x_1, \ldots, x_N)$ denote the joint decision vector, and define the mapping
$$F(x) = \big( A_1^\top \nabla f_1(A_1 x_1), \ldots, A_N^\top \nabla f_N(A_N x_N) \big).$$
Each player’s feasible set depends on the others’ strategies through the shared constraints $(B_i x_i + C_i x_{(-i)} = d_i)$, leading to the coupling feasible region
$$K_i(x_{(-i)}) = \{x_i \in X_i : B_i x_i + C_i x_{(-i)} = d_i, \ \forall i=\{1,\hdots,N\}\}.$$
The Nash equilibrium conditions then correspond to finding $x^* \in K(x^*)=\Pi_{i}K_i(x_{(-i)}^*)$ such that
$\langle F(x^*), x - x^* \rangle \ge 0, \quad \forall x \in K(x^*),$
which is a QVI problem. Inspired by Necoara, Nesterov,
and Glineur [2], we show that the QVI reformulation satisfies QG property if there exits $\eta_i>0$ such that $B_i^TB_i\preceq \eta_i A_i^TA_i$ for any $i$. 

First we show that for each $i$ there exists
a unique pair ($t_i^*, T_i^*)$ such that $A_ix_i^*=t_i^*$ and $A_i^\top \nabla f_i(A_i x_i^*)=T_i^*$. Let $x_i^{*1},x_i^{*2}$ be two optimal strategy of player $i$, and define $\phi_i(x_i,x_{(-i)})\triangleq f_i(A_ix_i)+c_i^Tx^*_{(-i)}$, then using convexity of $\phi_i(\cdot,x_{(-i)})$ and the fact that $x_i^*$ is optimal, i.e., $\phi_i(x_i^*,x^*_{(-i)})\leq \phi_i(x_i,x^*_{(-i)})$ for all $x_i\in K_i(x^*_{(-i)})$. It is also clear that $\tfrac{x_i^{*1}+x_i^{*2}}{2}$ is a feasible point, hence, we have that $\phi_i(\tfrac{x_i^{*1}+x_i^{*2}}{2},x^*_{(-i)})=\tfrac{\phi_i(x_i^{*1},x^*_{(-i)})}{2}+\tfrac{\phi_i(x_i^{*2},x^*_{(-i)})}{2}$, which implies that $f_i(\tfrac{A_ix_i^{*1}+A_ix_i^{*2}}{2})=\tfrac{f_i(A_ix_i^{*1})+f_i(A_ix_i^{*2})}{2}$. On the other hand, since $f_i$ is strongly monotone with modulus $\mu_i$, we have that $$f_i(\tfrac{A_ix_i^{*1}+A_ix_i^{*2}}{2})\leq \tfrac{f_i(A_ix_i^{*1})}{2}+\tfrac{f_i(A-ix_i^{*2})}{2}-\tfrac{\mu_i}{8}\|A_ix_i^{*1}-A_ix_i^{*2}\|^2.$$ Hence, we conclude that $A_ix_i^{*1}=A_ix_i^{*2}$.

Therefore, for each player $i$, the optimal solution set can be charactrized as $X_i^*(x^*_{(-i)})=\{x_i\mid A_ix_i=t_i^*, B_ix_i=d_i-C_ix_{(-i)}^*, D_ix_i\leq q_i\}$. Since $X_i^*(x^*_{(-i)})$ is a polyhedron and is nonempty, then from Hoﬀman inequality we have that there exists some positive constant depending on the matrices $A_i$, $B_i$ and $D_i$
describing the polyhedral set $X_i^*(x^*_{(-i)})$, i.e., $\theta_i (A_i, B_i,D_i) > 0$, such that for any $x_i\in X_i$:
\begin{align*}\|x_i-\bar x_i\|^2&\leq \theta_i^2(A_i,B_i,D_i)\left(\|A_ix_i-A_i\bar x_i\|^2+\|B_ix_i-B_i\bar x_i\|^2\right)\\&\leq \eta_i \theta_i^2(A_i,B_i,D_i)\|\|A_ix_i-A_i\bar x_i\|^2,\end{align*}
where $\bar x_i=\mathcal P_{X_i^*(x^*_{(-i)})}(x_i)$ and we used the assumption that $B_i^TB_i\preceq \eta_i A_i^TA_i$. Using this result, one can show that gradient of each player objective satisfies quadratic growth property:
\begin{align*}
    \langle A_i^T\nabla f_i(A_ix_i)-A_i^T\nabla f_i(A_i\bar x_i),x_i-\bar x_i \rangle&=\langle \nabla f_i(A_i x_i)-\nabla f_i(A_i\bar x_i),A_ix_i-A_i\bar x_i\rangle\\&\geq \mu_i\|A_ix_i-A_i\bar x_i\|\\&\geq\tfrac{\mu_i}{\sqrt{\eta_i}\theta_i(A_i,B_i,D_i)}\|x_i-\bar x_i\|.
\end{align*}
Therefore, summing across $i$, the full operator $F$ satisfies:
  $ \langle F(x) - F(\bar x), x - \bar x \rangle \geq \mu \|x - \bar x\|^2
  $ where $\mu=\min_i\{\tfrac{\mu_i}{\sqrt{\eta_i}\theta_i(A_i,B_i,D_i)}\}$ and $\bar x\in X^*=\Pi_{i=1}^N X_i^*(x^*_{(-i)})$.

\subsection{Assumptions and Technical Lemmas}
\aj{Now, we state our main assumptions considered in the paper.}
{\begin{assumption}\label{lipschitz}
\aj{(i) The set of optimal solution $X^*$ is nonempty.
(ii) Operator $F:X\rightarrow \mathbb R^n$ is monotone, i.e., $\langle F(x)-F(y),x-y\rangle\geq 0$ for all $x,y\in X$, and satisfies the QG property.
(iii) $F$ is $L$-Lipschitz continuous on $X$, i.e., 
\begin{align*}
\|F(x)-F(y)\|\leq L\|x-y\|,\quad \forall x, y \in X.
\end{align*}}
\end{assumption}}



If $\mathcal F_{k}$ denotes the information history at epoch $k$, then we have the following requirements on the associated filtration
{which imposes unbiasedness and bounded variance conditions on the stochastic error in estimating $F(x_k)$ using a finite sample average.}

\begin{assumption}\label{assump_error}
 There exists $\nu>0$ such that $\mathbb E[\bar w_{k,N_{k}}\mid  \mathcal F_{k}]=0$ and $\mathbb E[\| \bar w_{k,N_{k}}\|^2\mid  \mathcal F_{k}]\leq \tfrac{\nu^2}{N_{k}}$  holds almost surely 
 for all $k$, 
where $\mathcal{F}_k \triangleq \sigma\{x_0, x_1, \hdots, x_{k-1}\}$ and $\bar w_{k,N_k} \triangleq \tfrac{1}{N_k}{\sum_{j=1}^{N_k} \zal{( G(x_k,\xi_{j,k})-  F(x_k))}}$.  
\end{assumption} 

In our analysis, the following technical lemma for projection mappings {is used}.
\begin{lemma}\label{lem1}\cite{bertsekas2003convex}
\noindent Let $ X\subseteq \mathbb{R}^n $  be a nonempty closed and convex set. Then the following hold:
(a) $\|\mathcal{P}_X [u]- \mathcal{P}_X [v]\| \leq \|u-v\| $ for all $ u,v \in \mathbb{R}^n$;
(b) $ (\mathcal{P}_X [u]-u)^T(x-\mathcal{P} _X [u]) \geq 0 $ for all $u \in \mathbb{R}^n$ and $x \in X$.  
\end{lemma}
{Next, we define a gap function to measure the quality of the solution obtained from the algorithm.
\begin{definition}[Gap function.]\label{gap func}
    For a given iterate $x$ we use $\text{dist}(x,X^*)\triangleq \|x-\mathcal P_{X^*}(x)\|$ to find the distance of the solution obtained by the algorithm from the optimal solution set $X^*$. Moreover, we call $x$ to be an $\epsilon$-solution if $\|x-\bar x\|\leq \epsilon$ where $\bar x\in \mathcal P_{X^*}(x)$. 
\end{definition}}

\section{Proposed Method}\label{prop}

\aj{A popular method for solving SVI problems is the stochastic Extra-gradient (SEG) method originally proposed by Korpelevic \cite{korpelevich1976extragradient} for deterministic setting. In particular, when $K(x)=K$ is a closed and convex set, the \eqref{prob:SQVI} problem reduces to SVI. In each iteration of SEG, two consecutive projection steps are calculated with the following updates:
\begin{align*}
&u_k\gets \mathcal P_{K}\left(x_k-\eta \frac{\sum_{j=1}^{N_k}G(x_k,\xi_{j,k})}{N_k} \right);\\
&x_{k+1}\gets \mathcal P_{K}\left(u_k-\eta \frac{\sum_{j=1}^{N_k}G(u_k,\xi_{j,k})}{N_k} \right);
\end{align*}
The challenge in developing algorithms for solving SQVI primarily arises from the dynamic nature of the constraint set, which evolves during iterations. To manage this variation, it is crucial to ensure that the set $K(x)$ does not change drastically as $x$ varies. To achieve this, we impose a condition on the corresponding projection operator, which guarantees that the projection remains contractive with respect to the set $K(x)$.  This assumption is fundamental for convergence in QVI problems and is present in all existing results, indicating its necessity for current approaches \cite{noor1994general,nesterov2006solving,alizadeh2022inexact}.
\begin{assumption}\label{gamma exist}
There exists $\gamma>0$ such that $\|{\mathcal{P}}_{K(x)} [u]- {\mathcal{P}}_{K(y)} [u]\| \leq \gamma \|x-y\| $ for all $ x,y,u \in X$.\\ 
\end{assumption}
In many important applications, the convex-valued set $K(x)$ is of the form $K(x)=m(x)+K$, where $m$ is a point-to-point mapping and $K$ is a closed convex set. In this case,
$${\mathcal{P}}_{K(x)}[u]={\mathcal{P}}_{m(x)+K}[u]=m(x)+{\mathcal{P}}_{K}[u-m(x)], \quad \forall x,u\in X.$$
If $m$ is a Lipschitz continuous with constant $\tilde \gamma$, then
\begin{align*}
    \|{\mathcal{P}}_{K(x)} [u]- {\mathcal{P}}_{K(y)} [u]\| &= \|m(x)-m(y)+{\mathcal{P}}_{K}[u-m(x)]-{\mathcal{P}}_{K}[u-m(y)]\|\\&\leq
    \|m(x)-m(y)\|+\|{\mathcal{P}}_{K}[u-m(x)]-{\mathcal{P}}_{K}[u-m(y)]\|\\
    &\leq 2\|m(x)-m(y)\|\leq 2\tilde \gamma {\|x-y\|}.
\end{align*}
This shows that Assumption \ref{gamma exist} holds. Another important class of problems satisfying Assumption \ref{gamma exist} is when the solution set $K(x)$ is a bounded polyhedral moving set with right-hand-side linearly dependent on $x$. In such cases, by Corollary 2 in \cite{bednarczuk2021lipschitz}, the projection map $\mathcal{P}_{K(x)}$ is Lipschitz continuous.
Furthermore, from an algorithmic perspective to ensure convergence guarantee due to changes in the constraint set, one approach is to introduce a retraction step \cite{mijajlovic2019gradient}, denoted as $(1-\alpha)x_k+\alpha u_k$ for some $\alpha\in[0,1]$. By integrating these ideas, we present the following variant of the SEG method for solving \eqref{prob:SQVI}.
}
\aj{
\begin{align*} 
&v_{k}\gets \mathcal P_{K(x_k)}\left(x_k-\eta \frac{\sum_{j=1}^{N_k}G(x_k,\xi_{j,k})}{N_k} \right);\\
&u_{k}\gets(1-b_k)x_k+b_kv_k\\
&y_k\gets\mathcal P_{K(u_k)}\left(u_k-\eta \frac{\sum_{j=1}^{N_k}G(u_k,\xi_{j,k})}{N_k} \right);\\
&x_{k+1}\gets(1-\alpha_k)x_k+\alpha_k y_k;
\end{align*}

It is worth noting that, the exact computation of the projection onto the constraint set can be computationally expensive or even infeasible in certain scenarios. To overcome this limitation, we propose utilizing approximation techniques, allowing us to obtain practical solutions even when precise projections are unattainable. In particular, we assume that the constraints are comprised of the nonlinear smooth function $g: X\times X\to \mathbb R^m$, i.e., $K(x)=\{y\in X\mid g(x,y)\leq 0\}$, such that $g(x,\cdot)$ is convex for any $x\in X$. To handle the nonlinear constraints, in Algorithm \ref{alg2}, we introduce an inexact Extra-gradient SQVI (iEG-SQVI) method. In this method, at each step of the algorithm, we approximately solve the projection using an inner algorithm denoted as $\mathcal{M}$, which operates for a specified number of inner iterations, denoted as $t_k$. To guarantee fast convergence, Algorithm $\mathcal{M}$ should satisfy the following property.  }
\begin{assumption}\label{assump:inner}
For any $x\in\mathbb R^n$, any closed and convex set $K\subseteq \mathbb R^n$, and an initial point $u_0$, $\mathcal M$ can generate an output $u\in \mathbb R^n$ such that $\|u-\tilde u\|^2\leq C/t^2$ for some $C>0$ satisfying $\tilde u=\mbox{argmin}_{y\in K}\{\tfrac{1}{2}\|y-x\|^2\}$. 
\end{assumption}
\aj{In the following remark, we discuss that several optimization methods satisfy the condition outlined in Assumption \ref{assump:inner}.}
\aj{\begin{remark}
   When constraint set $K(x)$ is represented by (non)linear convex constraints, in steps (1) and (3) of Algorithm \ref{alg2}, one needs to compute the projection operator inexactly, which involves solving a strongly convex optimization problem subject to (non)linear convex constraints. In the field of optimization, various methods have been developed to address such problems. A highly efficient class of algorithms for solving this problem is the first-order primal-dual scheme such as \cite{he2015mirror,malitsky2018proximal,hamedani2018primal}. This method ensures a convergence rate of $\mathcal O(1/t^2)$ in terms of suboptimality and infeasibility, where $t$ represents the number of iterations.
\end{remark}}
\begin{algorithm} \caption{inexact Extra-gradient SQVI (iEG-SQVI) }
\label{alg2}
{\bf Input:}  $x_0\in X$, $\eta>0 $, ${T > 0}$, $\{N_k\}_k$, $\{t_k\}_k$, $\aj{\{b_k\}_k}$, $\{\alpha_k\}_k$ and Algorithm \aj{$\mathcal M$} satisfying Assumption \ref{assump:inner};\\
{\bf for $k=0,\hdots T-1$ do}\\
\quad \mbox{(1)} \aj{Use Algorithm $\mathcal M$ with $t_k$ iterations to find an approximated solution $d_k$ of $$\min_{x\in K(x_k)}\left\|x-\left(x_k-\eta\frac{\sum_{j=1}^{N_k}G(x_k,\xi_{j,k})}{N_k}\right)\right\|^2;$$}\\
\quad \mbox{(2)} $u_{k}\gets (1-b_k)x_k+b_k d_k$;\\
\quad \mbox{(3)} \aj{Use Algorithm $\mathcal M$ with $t_k$ iterations to find an approximated solution $s_k$ of $$\min_{x\in K(u_k)}\left\|x-\left(u_k-\eta\frac{\sum_{j=1}^{N_k}G(u_k,\xi'_{j,k})}{N_k}\right)\right\|^2;$$}\\
\quad \mbox{(4)} $x_{k+1}\gets(1-\alpha_k)x_k+\alpha_k s_k$;\\
{\bf end for}\\
{{\bf Output:} $x_{k+1}$.}
\\
\end{algorithm}
\subsection{Convergence Analysis}
\aj{Next, we introduce a crucial lemma for our convergence analysis. As previously discussed, the problem is not strong monotone and may not possess a unique solution (see \cite{facchinei2014solving,facchinei2003finite}). Therefore, we define the gap function as $\text{dist}(x,X^*)\triangleq \|x-\bar x\|$, where $\bar x\in\mathcal P_{X^*}(x)$. Since the optimal solutions are not explicitly available, we need to express $\bar{x}$ based on its first-order optimality condition. This representation will be utilized in the subsequent convergence analysis of the algorithm.} 
\begin{lemma}\label{lem:proj-optimal}
\aj{Let $X^*$ denote the set of solutions of problem \eqref{prob:SQVI}. Then, for any $\bar x\in X^*$ and $\eta>0$ the following holds}
\begin{align}\label{opt cond}
\bar x = \mathcal{P}_{K(\bar x)}(\bar x-\eta F(\bar x)).
\end{align} 
\end{lemma}
\begin{proof}
Note that $\bar x\in X^*$ implies that for any $x\in K(\bar x)$ we have that $\langle F(\bar x), x-\bar x \rangle \geq 0$. Multiplying both sides of the last inequality by $\eta>0$ we obtain that for any $x\in K(\bar x)$, $\langle \eta F(\bar x), x-\bar x \rangle \geq 0$ which due to convexity of set $K(\bar x)$ is equivalent to $\bar x=\mbox{argmin}_{x\in K(\bar x)}\|x-(\bar x-\eta F(\bar x))\|^2=\mathcal{P}_{K(\bar x)}(\bar x - \eta F(\bar x))$. 
\end{proof}
\aj{Before stating our main results, we define a few notations to facilitate the rate results.}
\z{
\begin{definition}
    At each iteration $k\geq 0$, \aj{we define the error of sample operator } $F$ as $$\bar w_{k,N_k} \triangleq \tfrac{1}{N_k}{\sum_{j=1}^{N_k} ( G(x_k,\xi_{j,k})-F(x_k))}\ \mbox{and} \ \bar w'_{k,N_k} \triangleq \tfrac{1}{N_k}\sum_{j=1}^{N_k} ( G(u_k,\xi_{j,k})- F(u_k)).$$ \aj{ Moreover, the error of approximating the projection is defined by} $$e_k\triangleq d_k- \mathcal P_{K(x_k)}\left(x_k-\eta \frac{\sum_{j=1}^{N_k}G(x_k,\xi_{j,k})}{N_k} \right)\ \mbox{and}$$   $$e'_k\triangleq s_k- \mathcal P_{K(u_k)}\left(u_k-\eta \frac{\sum_{j=1}^{N_k}G(u_k,\xi_{j,k})}{N_k} \right)$$ \aj{in step (1) and (3) of Algorithm \ref{alg2}, respectively.} 
    
\end{definition}
}
\aj{In the next theorem, we establish a bound on the expected solution error, which is expressed in terms of errors associated with the sample operator and the projection approximations. Subsequently, in Corollary \ref{corr alg2}, we provide the rate and complexity statements for Algorithm \ref{alg2}.}
\begin{theorem} \label{Quad eq:rate-general}
\aj{Let $\{x_k\}_{k\geq 0}$ be the iterates generated by Algorithm \ref{alg2} using step-size $\eta>0$ satisfying $|\eta-\tfrac{\mu_F}{L^2}|<{\frac{\sqrt {\mu_F^2-L^2(2\gamma-\gamma^2)}}{L^2}}$ and retraction parameters $\alpha_k=\bar\alpha\in (0,1)$ and $b_k=\bar b\in (0,\frac{1}{1-\beta})$ for $k\geq 0$, where $\beta\triangleq\gamma+\sqrt{1+L^2\eta^2-2\eta \mu_F}$. Suppose Assumptions \ref{lipschitz}-\ref{gamma exist} hold and $\gamma+\sqrt{1-{\mu_F}^2/L^2}<1$, then for any $T\geq 1$ we have that}
\begin{align}
 \nonumber \| x_{T}-\bar x_T \|&\quad \leq (1-q)^{T} \|x_{0}-\bar x_0\|+
{\bar\alpha \beta \bar b \sum_{k=0}^{T-1}(1-q)^{T-k-1}\left(\|e'_k\|+\eta\|\bar w'_{k,N_k}\|\right)}\\ &\qquad+{\bar\alpha\sum_{k=0}^{T-1}(1-q)^{T-k-1}\left(\|e_k\|+\eta\|\bar w_{k,N_k}\|\right)},\label{th1 result}
\end{align}
where $q\triangleq \bar\alpha(1-\beta)(1+\beta\bar b)\in (0,1)$. 
\end{theorem}
\begin{proof}
    
\aj{For any $k\geq 0$, we define $\bar x_k \in \mathcal{P}_{X^*}(x_k)$ where $X^*$ denotes the set of solutions of problem \eqref{prob:SQVI}. From Lemma \ref{lem:proj-optimal} we conclude that $\bar x_k= \mathcal{P}_{K(\bar x_k)}[\bar x_k-\eta F(\bar x_k)]$.} Using the update rule of $x_{k+1}$ in Algorithm \ref{alg2} and the fact that $e_k$ denotes the error of computing
the projection operator, we obtain the following.  
\begin{align}\label{Quadratic:b1}
\|x_{k+1}-\bar x_k\|&=\|(1-\alpha_k)x_k+\alpha_k \mathcal{P}_{K(u_k)}\left[u_k-\eta(F(u_k)+\bar w_{k,N_k})\right]\\
\nonumber &\quad+{\alpha_k e_k}-(1-\alpha_k)\bar x_k-\alpha_k\mathcal{P}_{K(\bar x_k)}\left[\bar x_k-\eta F(\bar x_k)\right]\| \\
\nonumber \quad&=\|(1-\alpha_k)x_k+\alpha_k \mathcal{P}_{K(u_k)}\left[u_k-\eta(F(u_k)+\bar w_{k,N_k})\right]\\
\nonumber &\quad+{\alpha_k e_k}-(1-\alpha_k)\bar x_k-\alpha_k\mathcal{P}_{K(\bar x_k)}\left[\bar x_k-\eta F(\bar x_k)\right]\\
\nonumber &\quad \pm \alpha_k \mathcal{P}_{K(x_k)}\left[u_k-\eta(F(u_k)+\bar w_{k,N_k})\right]\| \\
\nonumber&\leq (1-\alpha_k)\|x_k-\bar x_k\|+\alpha_k\|\mathcal{P}_{K(u_k)}\left[u_k-\eta(F(u_k)+\bar w_{k,N_k})\right]\\
\nonumber&\qquad-\mathcal{P}_{K(\bar x_k)}\left[u_k-\eta(F(u_k)+\bar w_{k,N_k})\right]\| \\
\nonumber&\quad +\alpha_k\|\mathcal{P}_{K(\bar x_k)}\left[u_k-\eta(F(u_k)+\bar w_{k,N_k})\right]\\
\nonumber&\qquad-\mathcal{P}_{K(\bar x_k)}\left[\bar x_k-\eta F(\bar x_k)\right]\|+{\alpha_k\|e_k\|}\\
\nonumber&\leq (1-\alpha_k)\|(x_k-\bar x_k)\|+\alpha_k\gamma \|u_k-\bar x_k\|\\
 &\nonumber\quad+\alpha_k\underbrace{\|u_k-\bar x_k-\eta(F(u_k)-F(\bar x_k))\|}_{\text{term (a)}}+\alpha_k\eta\|\bar w_{k,N_k}\|+{\alpha_k\|e_k\|},
\end{align}
\aj{where the first inequality follows from the triangle inequality, and in the last inequality, we used Lemma \ref{lem1}-(a) and Assumption \ref{gamma exist}.
Next, we provide an upper bound for the term (a) in \eqref{Quadratic:b1} by using Definition \ref{Quadratic growth def} and Lipschitz continuity of operator $F$ as follows} 
\begin{align}\label{Quadratic bound_term_a}
(\text{term (a)})^2
\nonumber&=\|u_k-\bar x_k\|^2+\eta^2\|F(u_k)-F(\bar x_k)\|^2\\ \nonumber&\quad-2\eta\langle u_k-\bar x_k,F(u_k)-F(\bar x_k)\rangle\\
\nonumber&\leq (1+L^2\eta^2-2\eta \mu_F)\|u_k-\bar x_k\|^2\\ 
\implies \text{term(a)}&\leq \sqrt{1+L^2\eta^2-2\eta \mu_F}\|u_k-\bar x_k\|.
\end{align}
\aj{Combining \eqref{Quadratic:b1} and \eqref{Quadratic bound_term_a}, and defining $\beta\triangleq\gamma+\sqrt{1+L^2\eta^2-2\eta \mu_F}$ we obtain} 
\begin{align} \label{Quad rate_x}
\|x_{k+1}-\bar x_k\| &\leq (1-\alpha_k)\|x_k-\bar x_k\|+\alpha_k \beta \|u_k-\bar x_k\|+\alpha_k\eta\|\bar w_{k,N_k}\| \zal{+\alpha_k\|e_k\|.}
\end{align}
\aj{Next, we turn our attention to providing an upper bound for $\|u_k-\bar x_k\|$. In particular, using the update of} $u_k$ in \aj{Algorithm} \ref{alg2} \aj{by taking similar steps as \eqref{Quadratic:b1} and \eqref{Quadratic bound_term_a}, one can obtain:}
\begin{align}
\nonumber\|u_{k}-\bar x_k\|&=\|(1-b_k)x_k+b_k \mathcal{P}_{K(x_k)}\left[x_k-\eta(F(x_k)+\bar w'_{k,N_k})\right]\\
\nonumber &\quad +{b_k e'_k} -(1-b_k)\bar x_k-b_k\mathcal{P}_{K(\bar x_k)}\left[\bar x_k-\eta F(\bar x_k)\right]\| \\
\nonumber &\leq (1-b_k)\|x_k-\bar x_k\|+b_k \beta \|x_k-\bar x_k\|+b_k\eta\|\bar w'_{k,N_k}\| \zal{+b_k\|e'_k\|}\\
\nonumber&= (1-b_k(1-\beta))\|x_k-\bar x_k\|+b_k\eta\|\bar w'_{k,N_k}\| \zal{+b_k\|e'_k\|}.
\end{align}
Replacing the above inequality in \eqref{Quad rate_x}, and \aj{defining} $q_i\triangleq \alpha_i (1-\beta)(1\z{+}\beta b_i)$ we conclude that 
\begin{align*}
\nonumber\|x_{k+1}-\bar x_k\|&\leq  (1-\alpha_k)\|x_k-\bar x_k\|+\alpha_k \beta( (1-b_k(1-\beta))\|x_k-\bar x_k\|\\
\nonumber&\quad+b_k\eta\|\bar w'_{k,N_k}\|\zal{+b_k\|e'_k\|} )+\alpha_k\eta\|\bar w_{k,N_k}\| \zal{+\alpha_k\|e_k\|}\\
\nonumber &= (1-\alpha_k (1-\beta)(1\z{+}\beta b_k))\|x_k-\bar x_k\|+\alpha_k\eta \beta b_k \|\bar w'_{k,N_k}\|\\
\nonumber &\quad\zal{+\alpha_k \beta b_k\|e'_k\|}+\alpha_k\eta\|\bar w_{k,N_k}\| \zal{+\alpha_k\|e_k\|}.
\end{align*}

\aj{Now, from the fact that $\bar x_{k+1}=\mathcal{P}_{X^*}(x_{k+1})$ one can conclude that $\|x_{k+1}-\bar x_{k+1}\|\leq \|x_{k+1}-\bar x_k\|$. Therefore, for any $k\geq 0$}
\begin{align*}
\nonumber\|x_{k+1}-\bar x_{k+1}\| &\leq  \prod_{i=0}^k (1-q_{i})\|x_{0}-\bar x_0 \| \\
\nonumber &\quad+ \sum_{i=0}^{k}\left(\left(\prod_{j=i}^{k-1}(1-q_{j})\right) \alpha_{i} \beta b_i \left(\eta \|\bar{w'}_{i,N_{i}}\|+\|e'_i\|\right)\right)\\
&\quad+ \sum_{i=0}^{k}\left(\left(\prod_{j=i}^{k-1}(1-q_{j})\right) \alpha_{i} \left(\eta \|\bar{w}_{i,N_{i}}\|+\|e_i\|\right)\right),
\end{align*}
\aj{where we assume that the product is 1 when there are no terms in the multiplication, i.e.,  $\prod_{j=i}^{k-1}(1-q_{j+1})=1$ if $i>k-1$.}

\aj{From the condition of $\eta$, we  have that $\beta<1$. Moreover, choosing $b_k =\bar b < \tfrac{1}{1-\beta}$ and $\alpha_k = \bar \alpha < 1$ one can readily verify that $q_k=q=\bar \alpha (1-\beta)(1+\beta \bar b)<1$ for all $k\geq 0$. Therefore, the result immediately follows by using the fact that $\prod_{j=i}^{k-1}(1-q)=(1-q)^{k-i}. $} 
\end{proof}
\aj{In the next corollary, we consider increasing the sample size at each iteration to demonstrate a linear convergence rate and obtain the oracle complexity.}
\begin{corollary}[Increasing sample-size]\label{corr alg2}
Under the premises of Theorem \ref{Quad eq:rate-general}, by selecting the number of inner steps for algorithm $\mathcal M$ as $t_k=\tfrac{({k+1})\log^2(k+2)}{\rho^{k}}$ and choosing the number of sample sizes at iteration $k$ as $N_k=\lceil \rho^{-2k}\rceil$ where $\rho>1-q$, we obtain the following results.

\textbf{(i)} \z{For} any $T\geq 1$, $\mathbb E[\|x_{T}-\z{\bar x_T}\|]\leq \mathcal O(\rho^{T}).$

{\textbf(ii)} \aj{An $\epsilon$-solution $x_T$, i.e., $\mathbb E[|x_{T}-\bar x_T|]\leq \epsilon$, can be achieved within $T=\mathcal O(\log(1/\epsilon))$ iterations which requires $\sum_{k=0}^{T-1} N_k \geq \mathcal O(1/\epsilon^2)$ sample operator evaluations and $\sum_{k=0}^{T-1} t_k = \mathcal O(\frac{1}{\epsilon}\log(1/\epsilon))$ number of total inner iterations.}
\end{corollary}
\begin{proof}
    
{\textbf (i)} \aj{Taking expectation from both {sides} of \eqref{th1 result}, choosing $N_k=\lceil \rho^{-2k}\rceil$, and using Assumption \ref{assump_error}, one can obtain:}
\begin{align*}
 &\mathbb E\left[ \| x_{T}-\z{\bar x_T }\|\right] \\
\nonumber&\leq (1-q)^{T} \|x_{0}-\z{\bar x_0}\| \aj{ + \bar{\alpha} \beta \bar b (1-q)^{T-1}\sum_{k=0}^{T-1}\left(\eta\nu' (\frac{\rho}{1-q})^k  +\mathbb E[\|e'_k\|](1-q)^{-k}\right)} \\
\nonumber & \aj{\quad+ \bar{\alpha} (1-q)^{T-1}\sum_{k=0}^{T-1}\left(\eta\nu (\frac{\rho}{1-q})^k  +\mathbb E[\|e_k\|](1-q)^{-k}\right).}
\end{align*}
\aj{ Using the fact that $ \sum_{\ze{k}=0}^{T-1} (\z{\tfrac{\rho}{1-q}})^{\ze{k}} = \tfrac{1-(\tfrac{\rho}{1-q})^{T}}{1-\tfrac{\rho}{1-q}} $, we conclude that }
\begin{align*}
\mathbb E\left[ \| x_T-\z{\bar x_T} \|\right] 
& \leq \z{(1-q)}^{T} \|x_{0}-\z{\bar x_0}\|+ \bar{\alpha} \beta \bar b \eta \nu'  \frac{ \rho^{T}-(1-q)^{T}}{\rho+q-1} \\
&\quad +{\bar\alpha \beta \bar b (1-q)^{T-1} \sum_{k=0}^{\ze{T-1}}\left(\z{\mathbb E[\|e'_k\|]}\z{(1-q)}^{-k}\right)} \\
&\quad + \bar{\alpha} \eta \nu \frac{ \rho^{T}-(1-q)^{T}}{\rho+q-1} +{\bar\alpha (1-q)^{T-1}\sum_{k=0}^{\ze{T-1}}\left(\z{\mathbb E[\|e_k\|]}\z{(1-q)}^{-k}\right).}
\end{align*}
\z{Since $\rho \geq 1-q$ and $q\in (0,1)$, \aj{one can easily confirm that $-\frac{(1-q)^{T}}{\rho+q-1}<0$, hence} the following holds.}
\begin{align*}
&\mathbb E\left[ \| x_T-\z{\bar x_T} \|\right] \\
& \leq \z{(1-q)}^{T} \|x_{0}-\z{\bar x_0}\|+   \frac{ \bar{\alpha} \beta \bar b \eta \nu'\rho^{T}}{\rho+q-1}
+{\bar\alpha \beta \bar b \sum_{k=0}^{\ze{T-1}}\left(\mathbb E[\|e'_k\|]\aj{(1-q)^{T-1-k}}\right)}\\
&\quad +\frac{ \bar{\alpha} \eta \nu \rho^{T}}{\rho+q-1}   +{\bar\alpha \sum_{k=0}^{\ze{T-1}}\left(\mathbb E[\|e_k\|]\aj{(1-q)^{T-1-k}}\right)}\\
&\leq \rho^{T} \|x_{0}-\z{\bar x_0}\|+   \frac{ \bar{\alpha} \beta \bar b \eta \nu'\rho^{T}}{\rho+q-1}  +\bar\alpha \beta \bar b \sum_{k=0}^{\ze{T-1}}\left(\mathbb E[\|e'_k\|]\aj{\rho^{T-1-k}}\right) \\
&\quad+ \frac{ \bar{\alpha} \eta \nu \rho^{T}}{\rho+q-1} +\bar\alpha \sum_{k=0}^{\ze{T-1}}\left(\mathbb E[\|e_k\|]\rho^{T-1-k}\right).
\end{align*}
\z{According to the Assumption \ref{assump:inner}, Algorithm \z{$\mathcal M$} has a convergence rate of $C/t_k^2$ within $t_k$ inner steps. By selecting $t_k=\tfrac{({k+1})\log^2(k+2)}{\rho^{k}}$, we have that $\mathbb E[ \|{e_k}\|]\leq \tfrac{C}{t_k}=\tfrac{C\rho^{k}}{(k+1)\log^2(k+2)}$ and $\mathbb E[ \|{e'_k}\|]\leq \tfrac{C'\rho^{k}}{(k+1)\log^2(k+2)}$. These upper bounds are independent of $x_k$, so by using the tower property of expectation in the previous inequality, we obtain the following. }
\begin{align*}
\mathbb E\left[ \| x_T-\z{\bar x_T} \|\right] &\leq
  \rho^{T} \|x_{0}-\bar x_0\|+   \frac{ \bar{\alpha} \beta \bar b \eta \nu'\rho^{T}}{\rho+q-1}  +\bar\alpha \beta \bar b C' \aj{\rho^{T-1}} \sum_{k=0}^{\ze{T-1}}\z{\tfrac{1}{(k+1)\log^2(k+2)}} \\
&\quad+ \frac{ \bar{\alpha} \eta \nu \rho^{T}}{\rho+q-1} +\bar\alpha C \aj{\rho^{T-1}}\sum_{k=0}^{\ze{T-1}} \z{\tfrac{1}{(k+1)\log^2(k+2)}}.
\end{align*}
\z{By applying the Cauchy-Schwarz inequality, using the fact that $ D\triangleq\sum_{\aj{k=0}}^{\aj{\infty}}\break  \tfrac{1}{(k+1)\log^2(k+2)}\aj{\leq 3.39}$ and rearranging the terms, the desired result is obtained}:

\begin{align}\label{corr1 part i}
  \mathbb E\left[ \| x_T-\z{\bar x_T} \|\right] \leq \z{\rho}^{T} \|x_{0}-\z{\bar x_0}\|+\z{\rho^T \bar \alpha \eta(\tfrac{ \beta \bar b \nu'+\nu}{\rho+q-1})+\rho^{T-1} \bar \alpha D (\beta\bar b C'+C).} 
\end{align}
{\textbf (ii)} To compute an $\epsilon$-solution, i.e., $\mathbb E[\|x_{T}-\bar x_T\|]\leq \epsilon$, it follows from \eqref{corr1 part i} that $T= \log_{1/\rho}(\bar D/\epsilon)$ iterations is required, where $\bar D=  \|x_{0}-\z{\bar x_0}\| \break +\z{ \bar \alpha \eta(\tfrac{ \beta \bar b \nu'+\nu}{\rho+q-1})+  \tfrac{ \bar \alpha D (\beta\bar b C'+C)}{\rho}}    $. 
Moreover, in Algorithm \ref{alg2}, each iteration requires taking $t_k=\tfrac{({k+1})\log^2(k+2)}{\rho^{k}}$ inner steps of Algorithm $\mathcal M$. Therefore, the total number of inner iterations is
\begin{align*}\sum_{k=0}^{T-1} t_k=\sum_{k=0}^{T-1} \tfrac{({k+1})\log^2(k+2)}{\rho^{k}}&\leq {T}\log^2(T+1)\tfrac{(1/\rho)^{T}}{1/\rho-1}=\log_{1/\rho}{\bar D/\epsilon}.\end{align*}
Furthermore, the total number of sample operator evaluations can be obtained as follows: 
$$\sum_{k=0}^{T-1} N_k=\sum_{k=0}^{T-1}\lceil \rho^{-2k}\rceil\geq \frac{\rho^2}{1-\rho^2}\left(\frac{\bar D^2}{\epsilon^2}-1\right).$$  
\end{proof}
\aj{\begin{remark}
    The error bound derived in Theorem \ref{Quad eq:rate-general} and Corollary \ref{corr alg2} signify convergence rates concerning the error associated with the projection operator. To be specific, in Corollary \ref{corr alg2}, we characterized how quickly this error must decrease to ensure linear convergence. Conversely, in cases where the projection onto the constraint set is straightforward to compute, i.e., when $e_k=e_k'=0$ for all $k \geq 0$, and under the assumptions of Theorem \ref{Quad eq:rate-general} the expectation of solution error will be bound as follows:
        \begin{align*}
\| x_{T}-\bar x_T \|  &\leq (1-q)^{T} \|x_{0}-\bar x_0\|\\
&\quad+
{\bar\alpha \beta \bar b\eta \sum_{k=0}^{T-1}(1-q)^{T-k-1}\|\bar w'_{k,N_k}\|}+{\bar\alpha\eta\sum_{k=0}^{T-1}(1-q)^{T-k-1}\|\bar w_{k,N_k}\|}.
\end{align*}
    By choosing $N_k=\lceil \rho^{-2k}\rceil$ where $\rho>1-q$, Algorithm \ref{alg2} achieves a linear convergence rate, i.e., $\mathbb E[\|x_{T}-\z{\bar x_T}\|]\leq \mathcal O(\rho^{T})$.
\end{remark}}
\aj{Next, let's examine a scenario in which a constant (mini) batch-size sampling operator of $F$ is accessible during each iteration $k\geq 0$. This configuration arises in various contexts, including online optimization, where problem information becomes progressively available over time. (See Appendix Section~\ref{sec:extragrad} for the proof.)}

\begin{corollary}[Constant mini-batch]\label{corr2 alg2}Under the premises of Theorem \ref{Quad eq:rate-general} by choosing $t_k=\tfrac{({k+1})\log^2(k+2)}{(1-q)^k}$ and $N_k=N$, then the following results hold.

{\textbf (i)} For any $T\geq 1$
  \begin{align*}
  \mathbb E\left[ \| x_T-\z{\bar x_T} \|\right]\leq \mathcal O\left(\aj{(1-q)}^T+\frac{1}{q\sqrt N}\right).
\end{align*}

\aj{{\textbf (ii)} Let mini-batch size  $N= \mathcal O(1/(q^2\epsilon^2))$. An $\epsilon$-solution $x_T$, i.e., $\mathbb E[|x_{T}-\bar x_T|]\leq \epsilon$, can be achieved within   $T= \mathcal O(\frac{1}{q}\log(1/\epsilon))$ iterations which requires $NT=\mathcal O(\frac{1}{q^3\epsilon^2}\log(1/\epsilon))$ sample operator evaluations.}
\end{corollary}
\aj{In the following corollary, we demonstrate an improvement in the oracle complexity for the deterministic case, reducing it to 
$\mathcal O(\log(1/\epsilon))$. The proof is provided in Appendix Section~\ref{sec:extragrad} and follows a similar approach to that of Corollary \ref{corr alg2}, with $w$ and $w'$ both being zero.}
\begin{corollary}[Deterministic QVI]\label{corr3 alg2}
    Under the premises of Theorem \ref{Quad eq:rate-general} for the deterministic case, by selecting the number of inner steps for algorithm $\mathcal M$ as $t_k=\tfrac{({k+1})\log^2(k+2)}{\rho^{k}}$  where $\rho>1-q$, the following holds. \\
{\textbf (i)} For any $T\geq 1$, $\mathbb E[\|x_{T}-\z{\bar x_T}\|]\leq \mathcal O(\rho^T).$
{\textbf (ii)} To compute an $\epsilon$-solution, i.e., $\mathbb E[\|x_{T}-\bar x_T\|]\leq \epsilon$, \aj{the total number of operator calls is $\mathcal O(\log(1/\epsilon))$ and total number of  inner iterations is $\sum_{k=0}^{T-1}t_k=\mathcal O(\frac{1}{\epsilon}\log(1/\epsilon))$.} 
\end{corollary}
\subsection{Gradient Approach}
\aj{A natural question that arises after proposing an extra gradient method is whether one can achieve similar convergence results for a gradient method. Therefore, in this section, we shift our focus to gradient method for solving the \eqref{prob:SQVI} problem, \za{ i.e.} letting retraction parameter $b_k=0$ in Algorithm \ref{alg2}. In particular, we propose an inexact Gradient SQVI (iG-SQVI) method in Algorithm \ref{alg1}.}
\begin{algorithm}[htbp]
\caption{\z{inexact Gradient SQVI (iG-SQVI)}}
\label{alg1}
{\bf Input}: $x_0\in X$, $\eta>0$,  $T > 0$,  $\{N_k\}_k$, $\{t_k\}_k$, $\{\alpha_k\}_k$ and Algorithm $\z{\mathcal M}$ satisfying Assumption \ref{assump:inner}; \\
{\bf for $k=0,\hdots T-1$ do}\\
\mbox{(1)} Use Algorithm $\mathcal M$ with $t_k$ iterations and find an approximated solution $d_k$ of $$\min_{x\in K(x_k)}\left\|x-\left(x_k-\eta\frac{\sum_{j=1}^{N_k}G(x_k,\xi_{j,k})}{N_k}\right)\right\|^2;$$
\mbox{(2)} $x_{k+1}=(1-\alpha_k)x_k+\alpha_k d_k$;  \\
{\bf end for}\\
{\bf Output:} $x_{k+1}$\zr{.}
\end{algorithm}
\\

\aj{In the next theorem and corollary, we show the convergence results of Algorithm \ref{alg1} which is similar to the rate results of Algorithm \ref{alg2}. (See Appendix Section~\ref{sec:gradient} for the proof.)} 

\begin{theorem} \label{th: alg1}
\aj{Let $\{x_k\}_{k\geq 0}$ be the iterates generated by Algorithm \ref{alg1} using step-size $\eta>0$ satisfying $|\eta-\tfrac{\mu_F}{L^2}|<{\frac{\sqrt {\mu_F^2-L^2(2\gamma-\gamma^2)}}{L^2}}$ and retraction parameter $\alpha_k=\bar\alpha\in (0,1)$ for $k\geq 0$. Suppose Assumptions \ref{lipschitz}-\ref{gamma exist} hold and $\gamma+\sqrt{1-{\mu_F}^2/L^2}<1$, then for any $T\geq 1$ we have that}
\begin{align}
 & \| x_{T}-\bar x_T \| \leq (1-q)^{T} \|x_{0}-\bar x_0\|+{\bar\alpha\sum_{k=0}^{T-1}(1-q)^{T-k-1}\left(\|e_k\|+\eta\|\bar w_{k,N_k}\|\right)},
\label{th2 result}\end{align}
where $q\triangleq \bar\alpha(1-\beta)\in (0,1)$ and $\beta\triangleq\gamma+\sqrt{1+L^2\eta^2-2\eta \mu_F}$. 
\end{theorem}
\begin{corollary}\label{corr alg1}
    Under the premises of Theorem \ref{th: alg1}   by selecting the number of inner steps for algorithm $\mathcal M$ as $t_k=\tfrac{({k+1})\log^2(k+2)}{\rho^{k}}$ and choosing $N_k=\lceil \rho^{-2k}\rceil$ where $\rho>1-q$, \z{the following results can be 
obtained.}

    \textbf{(i)} For any $T\geq 1$, $\mathbb E[\|x_{T}-\z{\bar x_T}\|]\leq \mathcal O(\rho^{T}).$

{\textbf (ii)} To find an $\epsilon$-solution $x_T$, i.e., $\mathbb E[\|x_{T}-\bar x_T\|]\leq \epsilon$, the total number of sample operators and inner iterations of algorithm $\mathcal M$ are  $\sum_{k=0}^{T-1} N_k\geq \mathcal O(1/\epsilon^2)$ and $\sum_{k=0}^{T-1} t_k\leq\mathcal O(\frac{1}{\epsilon}\log(1/\epsilon))$, respectively.
\end{corollary}
\begin{remark}
\aj{Here, we would like to highlight the difference between the iEG-SQVI and iG-SQVI algorithms. According to Theorems \ref{Quad eq:rate-general} and \ref{th: alg1}, both algorithms demonstrate a linear convergence rate of $(1-q)^T$, where $q=\bar\alpha(1-\beta)(1+\beta\bar b)$ for iEG-SQVI and $q=\bar\alpha(1-\beta)$ for iG-SQVI. It is evident that the extra-gradient method (iEG-SQVI) achieves a faster convergence rate due to the additional factor of $(1+\beta\bar b)$. However, it is important to note that the extra-gradient method requires twice the number of operator and (inexact) projection evaluations at each iteration. Consequently, in cases where these evaluations are costly, one might prefer implementing the gradient method (G-SQVI); otherwise, the extra-gradient method (iEG-SQVI) may be more favorable.}

\end{remark}
Note that to obtain convergence rate guarantees, we required the condition $\gamma + \sqrt{1 - \frac{\mu_F^2}{L^2}} < 1$, which is problem-dependent and commonly appears in prior works on deterministic QVI problems with convergence guarantees (e.g., see Theorem 2 and 3 in \cite{mijajlovic2019gradient} and Theorem 5 in \cite{nesterov2006solving}). Whether convergence rate guarantees—such as sublinear or linear rates—can be established without relying on this condition remains an open question, particularly due to the added complexity introduced by feasible sets that depend on the decision variable. In the next section, we show that almost sure convergence can be achieved under a weaker requirement, which becomes less restrictive and easier to satisfy. 
\subsection{Almost Sure Convergence}
To establish almost sure convergence of the iterates, we first state the following Robbins-Siegmund lemma (see Lemma 11 in \cite{robbins1971convergence}).

\begin{lemma}[Robbins-Siegmund]\label{robins sig} Suppose $\omega_t,\eta_t,\upsilon_t
$ and $\psi_t$ are nonnegative random variables that satisfy:
$$
\mathbb{E}\left[\omega_{t+1} \mid \mathcal{F}_t\right] \leq\left(1+\eta_t\right) \omega_t+\upsilon_t-\psi_t, \quad \forall t\geq0.
$$
where $\mathbb{E}\left[\omega_{t+1} \mid \mathcal{F}_t\right]$ represents the conditional expectation for the given $\{\omega_0...,\omega_t\}$, $\{\eta_0,...,\eta_t\},$ $
\{\upsilon_0,...,\upsilon_t\}
$, $\{\psi_0,...,\psi_t\}$, and $\sum_{t=0}^\infty \eta_t < \infty$, $\sum_{t=0}^\infty \upsilon_t < \infty$ , then we have:  $
\text { (i) }  \omega_t \rightarrow \omega$ almost surely, and $\text { (ii) } \sum_{t=0}^\infty \psi_t < \infty$. 
\end{lemma}

Now we are ready to show that the
iterates generated by Algorithm \ref{alg1} have at least one limit point, and every limit point is a solution of the problem \ref{prob:SQVI} almost surely. (See Appendix Section~\ref{sec:a.s.} for the proof.)
\begin{theorem}[Almost sure convergence]\label{thm:as_convergence}
Let $\{x_k\}_{k\geq 0}$ be the sequence generated by Algorithm \ref{alg1}, and
suppose Assumption \ref{lipschitz} and \ref{assump_error} hold and $\{N_k\}_k$ is an increasing sequence, such that $\sum_{k=0}^\infty \tfrac{1}{{N_k}}< \infty$. Choose $\eta= \frac{1.8\mu_F}{L^2}$, and $\gamma\leq 45(\sqrt{1+0.04\frac{\mu_F^2}{L^2}}-1)$, and $\alpha_k=\tfrac{0.01}{(k+1)^{0.5+\delta}}$, for some $\delta>0$.Then the sequence $\{x_k\}$ has at least one limit point, and every limit point is a solution of the \ref{prob:SQVI} almost surely.
\end{theorem}
Note that the upper-bound for $\gamma$ in Theorem \ref{thm:as_convergence} is greater than the upper-bound in Theorem \ref{th: alg1}, i.e., $45(\sqrt{1+0.04\frac{\mu_F^2}{L^2}}-1)>1-\sqrt{1-\frac{\mu_F^2}{L^2}}$, when $\frac{\mu_F}{L}\leq0.992$, which is the case for most real-world problems, as $\frac{\mu_F}{L}$ is typically much smaller than $1$.
\section{Numerical Experiment}\label{numeric}
\textbf{Over-parameterized Regression Game.} In a regression problem, the goal is to find a parameter vector $x \in \mathbb{R}^d$ that minimizes the loss function $\ell^{\text{tr}}(x)$ over the training dataset $D^{\text{tr}}$. Without explicit regularization, an over-parameterized regression problem exhibits multiple global minima over the training dataset, and not all optimal regression coefficients perform equally well. Introducing a secondary objective provides a principled way to select among these solutions: it guides us toward parameters that not only fit the training data but also optimize an additional desirable property. For instance, using the $\ell_1$-norm as a secondary objective favors sparse solutions, while using the 
$\ell_2$-norm selects solutions with smaller overall magnitude.

\begin{figure}[htb]
    \centering
    \includegraphics[scale=0.1]{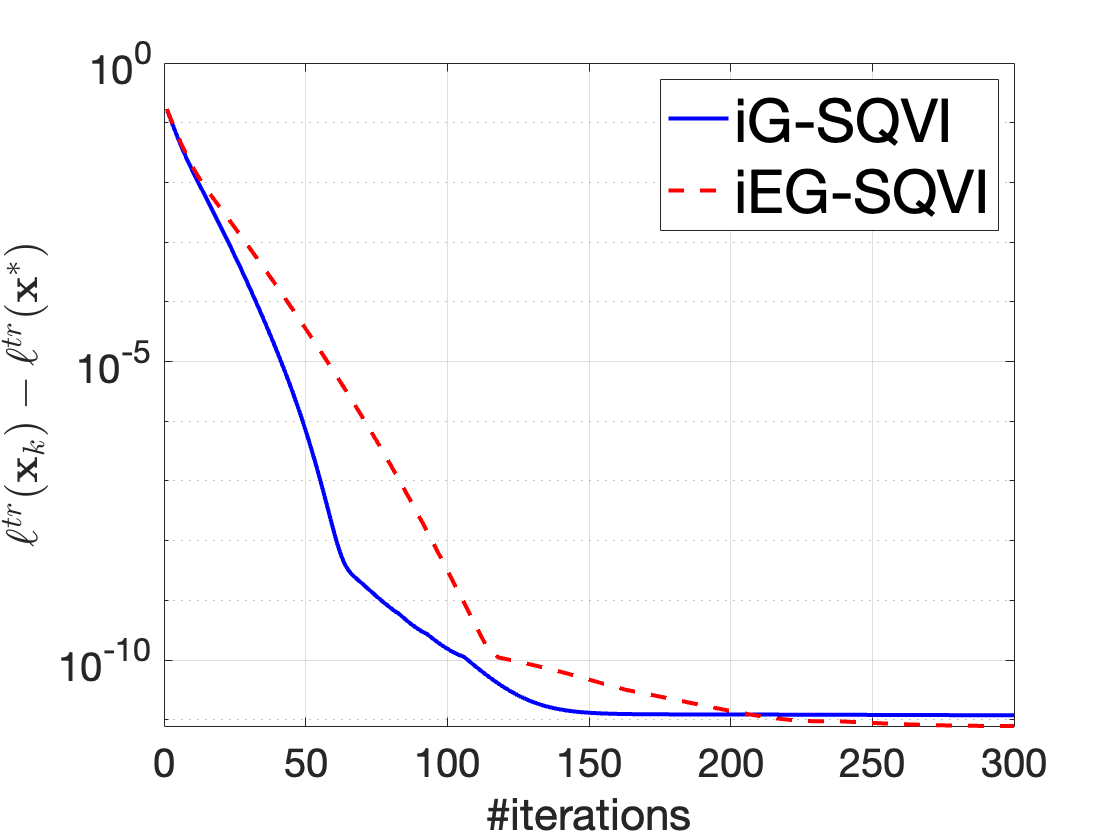}
    \includegraphics[scale=0.1]{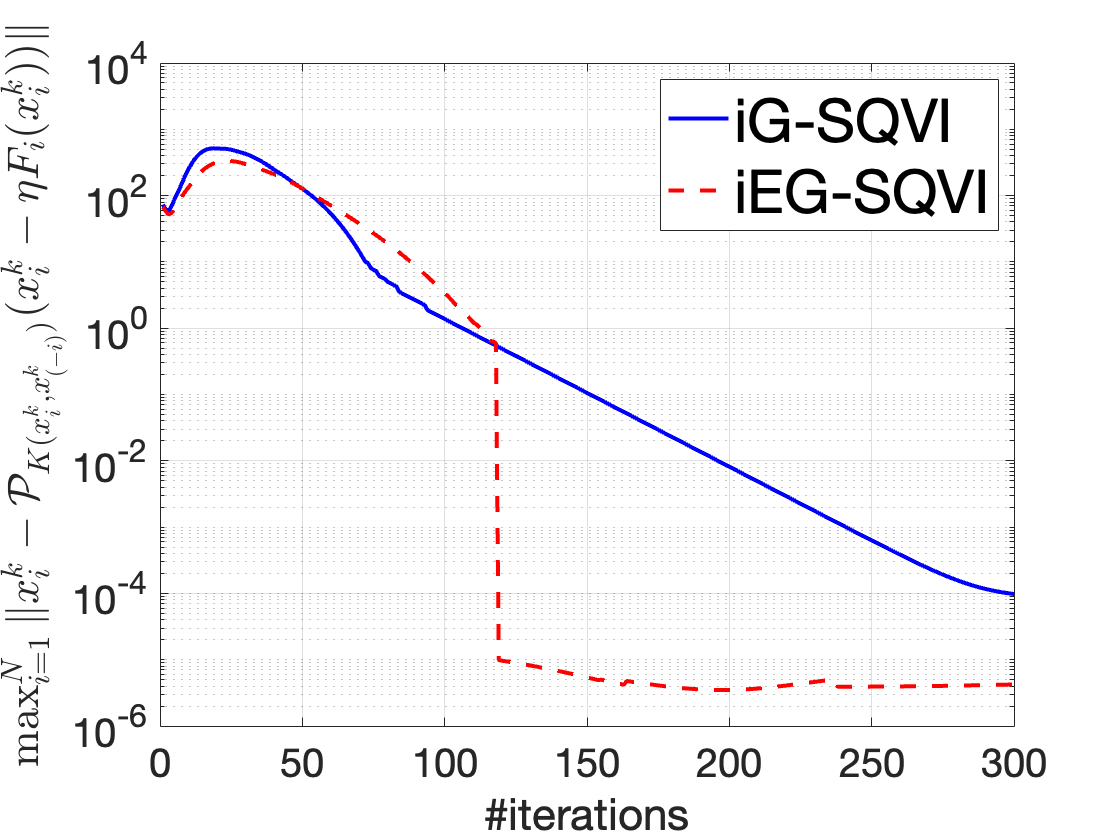}\\
    \includegraphics[scale=0.1]{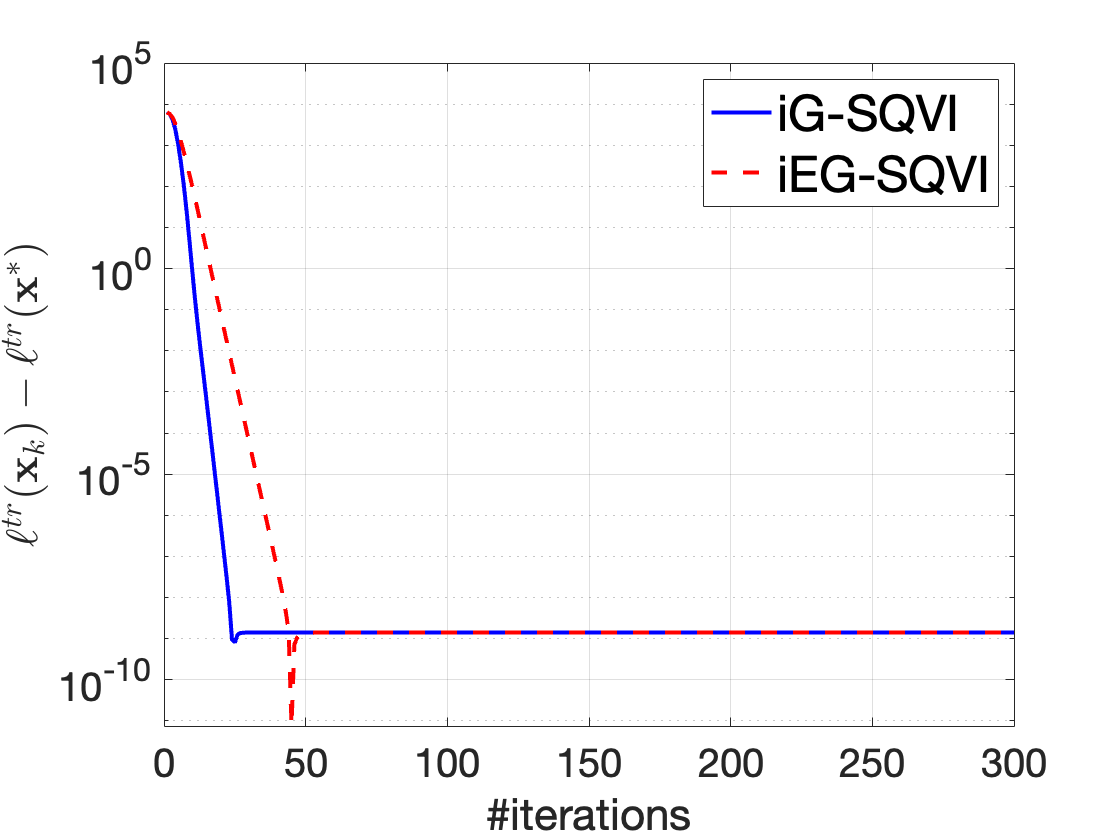}
    \includegraphics[scale=0.1]{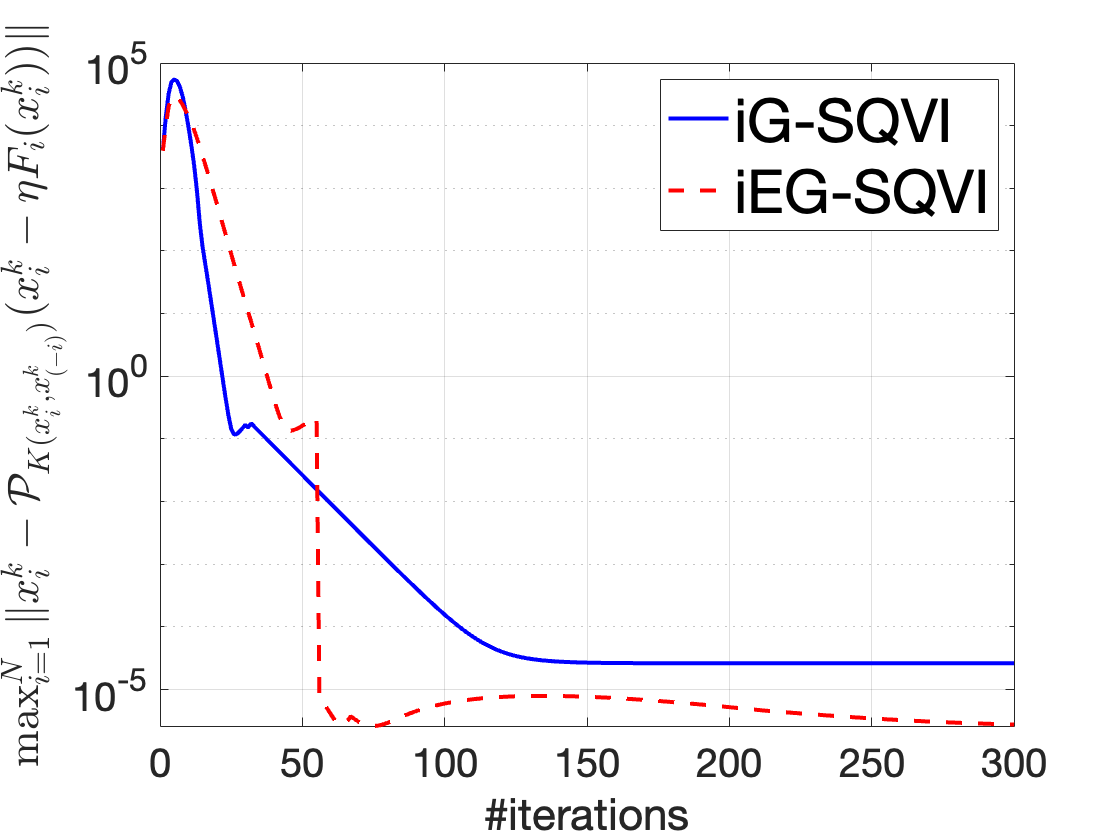}

    \caption{Comparison of iG-SQVI and iEG-SQVI: (Top) for the \texttt{wikivital mathematics} dataset with 300 data points and 1065 features, (Bottom) for a \texttt{synthetic} dataset with 60 data points and 500 features.}
    \label{fig:regression}
\end{figure}

Consider a collection of $N$ players each having a model parameter $x_i\in \mathbb R^{d_i}$. Define $\mathbf{x} \triangleq [x_i]_{i=1}^N$, and suppose there is a shared training dataset $D^{\text{tr}}$ 
and the goal is to find a model parameter $\mathbf x$ by minimizing the training loss $\ell^{\text{tr}}(\mathbf x)$ while each player improves its model parameter based on their secondary objective by minimizing a weighted $\ell_2$-norm. This problem can be formulated as the following bilevel GNE:
\begin{align*}
    &\min_{x_i\in  X_i} \ \ell^{W}_i(x_i) \quad \mbox{s.t.} \quad x_i\in \mbox{argmin}_{x_i\in \mathbb R^d} \ell^{\text{tr}}\left(x_i,x_{(-i)}^*\right).
\end{align*}
In this experiment, we define  
$\ell^{\text{tr}}(\mathbf x)\triangleq\frac{1}{2}\|A^{\text{tr}}\mathbf x-b^{\text{tr}}\|^2,$ 
where $A^{\text{tr}}\in \mathbb R^{n\times d}$, $d=\sum_{i=1}^N d_i$, $n<\min_i\{d_i\}$, $b^{\text{tr}}\in \mathbb R^{n\times 1}$, and $\ell^{W}_i(x_i)\triangleq \frac{1}{2}\|W_ix_i\|^2,$
where $W_i\in \mathbb R^{p_i\times d_i}$ is generated randomly such that $p_i<d_i$, and $X_i=\{x_i\mid \|x_i\|_{1}\leq \lambda\}$ for some $\lambda>0$. One can show that this problem can \zr{be} formulated as QVI by choosing $K(x)=\Pi_{i=1}^N K_i(x_{(-i)})$, where $$K_i(x_{(-i)})=\mbox{argmin}_{x_i\in \mathbb R^{d_i}} \frac{1}{2}\left\|A_i^{\text{tr}}x_i+A_{(-i)}^{\text{tr}}x_{(-i)}-b_i^{\text{tr}}\right\|^2,$$ and $F(x)=[F_i(x_i)]_{i=1}^N$ where $$F_i(x_i)=W_i^T(W_ix_i).$$

Note that, this problem belongs to the class of problem we defined in subsection \ref{sec:QG}. Operator $F$ is monotone and satisfies quadratic growth property but it is not be strongly monotone. Furthermore, one can show that Assumption \ref{gamma exist} is satisfied (see Corollary 2 in \cite{bednarczuk2021lipschitz}). In Figure \ref{fig:regression}, we present a performance comparison of our proposed methods. For the \texttt{wikivital mathematics} dataset \cite{rozemberczki2021pytorch}, we set the number of players $N=3$ and for the \texttt{synthetic} dataset $N=5$.  To solve the projection inexactly, observe that the sub-problem is a simple bilevel optimization problem. This type of problem has been explored in the literature \cite{jiang2023conditional,samadi2023achieving}. Here, following \cite{samadi2023achieving}, we employed the FISTA algorithm \cite{beck2009fast} to solve the corresponding regularized problem satisfying Assumption \ref{assump:inner}. For all the experiments, we execute the inner algorithm for $k\log^2(k+1)(1-10^{-4})^{-k}$ iterations, and the remaining parameters are selected according to the following table after fine-tuning.

\begin{table}[htb]
    \centering
    \caption{Parameter settings after fine-tuning for the algorithms across all datasets }
    \label{tab:param}
    \begin{tabular}{|c||c|c|}\hline
       & \texttt{wikivital mathematics}  &\texttt{synthetic} \\\hline\hline
        Stepsize $\eta$ &1e-4 &1e-4\\
        $\bar\alpha$&5e-2&1e-1\\
        $\bar b$&2.5e-2&1e-1\\
        Regularizer&1e-4&1e-4\\ \hline
    \end{tabular}
\end{table}

In Figure \ref{fig:regression}, on the left, we compared the suboptimality of the lower-level problem, and on the right, we compared the gap function based on the optimality condition \eqref{opt cond}. It is evident that both methods converge to the optimal solution. Notably, iEG-SQVI demonstrates a slightly better performance due to a smaller convergence rate factor.


\section{Conclusion}\label{summary}
\aj{This paper focuses on solving a class of monotone stochastic quasi-variational inequality problems where the operator satisfies quadratic growth property. We introduce extra-gradient and gradient-based schemes and characterize the convergence rate and oracle complexity of the proposed methods. To the best of our knowledge, our proposed algorithms are the first with a convergence rate guarantee when dealing with non-strongly monotone SQVI problems, especially when projecting onto the constraints is challenging. In our numerical experiments, we showcase the effectiveness and robustness of our methods in solving over-parameterized regression games. These results mark an important first step in exploring broader scenarios, including monotone and weakly-monotone SQVIs. Future directions also involve delving into distributed and risk-based SQVI problems.}\\

\noindent \small{{\bf Author contribution}\ \ The manuscript was mainly written and edited by ZA. AJ provided research conditions and
guidance. The technical implementations,
such as coding algorithms and setting up reproducible environments were done by ZA.
All authors read and approved the final manuscript.}\\

\noindent \small{{\bf Funding}\ \ This work is supported in part by the National Science Foundation under Grant ECCS-2231863, the University of Arizona Research, Innovation \& Impact (RII) Funding, and the Arizona
Technology and Research Initiative Fund (TRIF) for Innovative Technologies for the Fourth Industrial
Revolution initiatives.}\\

\noindent \small{{\bf Availability of data and materials}\ \  The basic code of this work are
available from the corresponding author upon reasonable request.}

\bibliographystyle{abbrv}      
\bibliography{biblio}

\appendix
\section*{Appendix}
\setcounter{equation}{0} 
\renewcommand{\theequation}
{A\arabic{equation}}
\section{Extra Gradient Method}\label{sec:extragrad}
{\bf Proof of Corollary \ref{corr2 alg2}.} 
\z{{\textbf (i)} By taking expectation from both sides of \eqref{th1 result}, choosing $N_k =N$, and using Assumption \ref{assump_error}, the following holds.}
    \begin{align}
 \nonumber\mathbb E\left[ \| x_{T}-\z{\bar x_T }\|\right]&\leq  (1-q)^{T} \|x_{0}-\z{\bar x_0}\|\aj{+ \bar{\alpha} \beta \bar b (1-q)^{T-1}\sum_{k=0}^{T-1}\left( \frac{\eta\nu' }{(1-q)^{k}\sqrt{N}} +\mathbb E[\|e'_k\|](1-q)^{-k}\right)} \\
\nonumber & \quad \aj{+ \bar{\alpha} (1-q)^{T-1}\sum_{k=0}^{T-1}\left(\frac{\eta\nu }{(1-q)^k \sqrt{N}}  +\mathbb E[\|e_k\|](1-q)^{-k}\right).}
\end{align}
\aj{Following the similar steps as in the proof of Corollary \ref{corr alg2},}
the following can be obtained by defining $ D\triangleq\sum_{\aj{k=0}}^{\aj{\infty}} \tfrac{1}{(k+1)\log^2(k+2)}\aj{\leq 3.39}$.
\begin{align*}
    &\mathbb E\left[ \| x_T-\z{\bar x_T} \|\right]  \leq \z{(1-q)}^{T} \|x_{0}-\z{\bar x_0}\|+  \frac{\bar{\alpha} \beta \bar b \eta \nu' }{q\sqrt{N}}+  \frac{ \bar{\alpha} \eta \nu}{q\sqrt{N}}    +{\bar\alpha \beta \bar b C'D (1-q)^{T-1}  }  + {\bar\alpha C D (1-q)^{T-1}}.
\end{align*}
Now by rearranging the terms, we obtain the desired result: 
\begin{align}\label{corr 2 part i}
    &\mathbb E\left[ \| x_T-\z{\bar x_T} \|\right]  \leq \z{(1-q)}^{T} \|x_{0}-\z{\bar x_0}\| + (1-q)^{T-1}\bar\alpha D { (\beta \bar b C'+C )   }   +  \frac{\bar{\alpha} \eta   }{q\sqrt{N}}(\beta \bar b \nu'+\nu). 
\end{align} 
{\textbf (ii)} Let $T= \log_{1/(1-q)}(2\bar D/\epsilon)$, $N = \tfrac{4\bar C^2}{q^2 \epsilon^2}$, and define $\bar D \triangleq \|x_{0}-\z{\bar x_0}\| +\tfrac{\bar\alpha D { (\beta \bar b C'+C )}}{(1-q)}  $ and $\bar C \triangleq \bar{\alpha} \eta (\beta \bar b \nu'+\nu)$, then from  \ref{corr 2 part i} we have that:
\begin{align*}
  \nonumber  \mathbb E\left[ \| x_T-\z{\bar x_T} \|\right] &\leq\z{(1-q)}^{T}\big( \|x_{0}-\z{\bar x_0}\| + (1-q)\bar\alpha D { (\beta \bar b C'+C ) \big)  }  +  \frac{\bar{\alpha} \eta   }{q\sqrt{N}}(\beta \bar b \nu'+\nu)\\
 &\leq\z{(1-q)}^{T}\bar D +  \frac{\bar C}{q\sqrt{N}}\leq \epsilon,
\end{align*}
where in the last inequality we used the definition of $T$ and $N$.$\square$

 {\bf Proof of Corollary \ref{corr3 alg2}.} \aj{\textbf{(i)} When a problem is deterministic, we have that $w=w'=0$. Therefore, from inequality \ref{th1 result} one can obtain: }     

     \begin{align*}
\nonumber & \| x_{T}-\bar x_T \| \leq (1-q)^{T} \|x_{0}-\bar x_0\|+{\bar\alpha\sum_{k=0}^{T-1}(1-q)^{T-k-1}\|e_k\|}+
{\bar\alpha \beta \bar b \sum_{k=0}^{T-1}(1-q)^{T-k-1}\|e'_k\|},
\end{align*}
where $q\triangleq \bar\alpha(1-\beta)(1+\beta\bar b)\in (0,1)$. Now, taking expectations from both sides of the previous inequality, the following holds:
\begin{align*}
   \mathbb E\left[ \| x_T-\z{\bar x_T} \|\right] 
& \leq \z{(1-q)}^{T} \|x_{0}-\z{\bar x_0}\| +{\bar\alpha \beta \bar b (1-q)^{T-1} \sum_{k=0}^{\ze{T-1}}\left(\z{\mathbb E[\|e'_k\|]}\z{(1-q)}^{-k}\right)} \\
&\quad +{\bar\alpha (1-q)^{T-1}\sum_{k=0}^{\ze{T-1}}\left(\z{\mathbb E[\|e_k\|]}\z{(1-q)}^{-k}\right).}
\end{align*}
\aj{Following the similar steps as in the proof of Corollary \ref{corr alg2}, the following can be obtained:}
\begin{align*}
   \nonumber&\mathbb E\left[ \| x_T-\z{\bar x_T} \|\right] \leq \z{\rho}^{T} \|x_{0}-\z{\bar x_0}\|+\rho^{T-1} \bar \alpha D(\beta\bar b C'+C).
\end{align*}
\textbf{(ii)} Similar steps as in the proof of Corollary \ref{corr alg2} (ii). $\square$
\section{Gradient Method}\label{sec:gradient}
\aj{In this section, we prove the convergence results of iG-SQVI algorithm. In our analysis, $e_k$ denotes the error of computing the projection operator, i.e., for any $k\geq 0$, \z{$   e_k\triangleq d_k- \mathcal P_{K(x_k)}\left(x_k-\eta \frac{\sum_{j=1}^{N_k}G(x_k,\xi_{j,k})}{N_k} \right) $ and we define $\bar w_{k,N_k} \triangleq \tfrac{1}{N_k}{\sum_{j=1}^{N_k} ( G(x_k,\xi_{j,k})-F(x_k))}$.}
}

{\bf Proof of Theorem \ref{th: alg1}.}
For any $k\geq 0$, we define $\bar x_k \zr{\in} \mathcal{P}_{X^*}(x_k)$ where $X^*$ denotes the set of optimal solutions of problem \eqref{prob:SQVI}. From Lemma \ref{lem:proj-optimal} we conclude that $\bar x_k= \mathcal{P}_{K(\bar x_k)}[\bar x_k-\eta F(\bar x_k)]$. Using the update rule of $x_{k+1}$, in Algorithm \ref{alg1} and the fact that $e_k$ denotes the error of computing
the projection operator, we obtain the following. 
\begin{align}\label{b1}
\nonumber\|x_{k+1}-\bar x_k\|&=\|(1-\alpha_k)x_k+\alpha_k \mathcal{P}_{K(x_k)}\left[x_k-\eta(F(x_k)+\bar w_{k,N_k})\right]\\
\nonumber &\quad+{\alpha_k e_k}-(1-\alpha_k)\bar x_k-\alpha_k\mathcal{P}_{K(\bar x_k)}\left[\bar x_k-\eta F(\bar x_k)\right]\| \\
\nonumber&\leq \|(1-\alpha_k)(x_k-\bar x_k)\|\\
\nonumber&\quad+\alpha_k\|\mathcal{P}_{K(x_k)}\left[x_k-\eta(F(x_k)+\bar w_{k,N_k})\right]\\
\nonumber&\qquad-\mathcal{P}_{K(\bar x_k)}\left[x_k-\eta(F(x_k)+\bar w_{k,N_k})\right]\| \\
\nonumber&\quad +\alpha_k\|\mathcal{P}_{K(\bar x_k)}\left[x_k-\eta(F(x_k)+\bar w_{k,N_k})\right]\\
\nonumber&\qquad-\mathcal{P}_{K(\bar x_k)}\left[\bar x_k-\eta F(\bar x_k)\right]\|+{\alpha_k\|e_k\|}\\
\nonumber&\leq \|(1-\alpha_k)(x_k-\bar x_k)\|+\alpha_k\gamma \|x_k-\bar x_k\|\\
&\quad+\alpha_k\underbrace{\|x_k-\bar x_k-\eta(F(x_k)-F(\bar x_k))\|}_{\text{term (a)}}+\alpha_k\eta\|\bar w_{k,N_k}\|+{\alpha_k\|e_k\|.}
\end{align}
By using Definition \ref{Quadratic growth def} and Lipschitz continuity, the following can be obtained, 
\begin{align}\label{bound_term_a}
\nonumber\|x_k-\bar x_k-\eta(F(x_k)-F(\bar x_k))\|^2&=\|x_k-\bar x_k\|^2+\eta^2\|F(x_k)-F(\bar x_k)\|^2\\
\nonumber&\quad-2\eta\langle x_k-\bar x_k,F(x_k)-F(\bar x_k)\rangle\\
\nonumber&\leq (1+L^2\eta^2-2\eta \mu_F)\|x_k-\bar x_k\|^2\\& \implies \text{term(a)}\leq \sqrt{1+L^2\eta^2-2\eta \mu_F}\|x_k-\bar x_k\|.
\end{align}
Now by using \eqref{bound_term_a} in \eqref{b1}, defining $\beta\triangleq\gamma+\sqrt{1+L^2\eta^2-2\eta\mu_F}$ and $q_i\triangleq (1-\beta)\alpha_i$ we get the following:
\begin{align*}
\nonumber&\|x_{k+1}-\bar x_k\| \leq\\
\nonumber&\quad(1-\alpha_k)\|x_k-\bar x_k\|+\alpha_k\eta\|\bar w_{k,N_k}\| \zal{+\alpha_k\|e_k\|}\\
\nonumber&\qquad+\alpha_k\left(\gamma+\sqrt{1+L^2\eta^2-2\eta\mu_F}\right)\|x_k-\bar x_k\|\\
\nonumber&=(1-(1-\beta)\alpha_k)\|x_k-\bar x_k\|+\alpha_k\eta\|\bar w_{k,N_k}\|+{\alpha_k\|e_k\|}\\
\nonumber&\leq { \prod_{i=0}^k (1-q_{i})\|x_{0}-\bar x_k \| + \alpha_{k} \left(\eta \|\bar{w}_{k,N_{k}}\|+\|e_k\|\right)}\\
&\qquad + \sum_{i=0}^{k-1}\left(\left(\prod_{j=i}^{k-1}(1-q_{j+1})\right) \alpha_{i} \left(\eta \|\bar{w}_{i,N_{i}}\|+\|e_i\|\right)\right).
\end{align*} 
Next, by choosing $\alpha_k=\bar \alpha$, where $0<\bar\alpha<1$, and based on the conditions of the theorem, one can easily verify that $\beta<1$ and  $q_k=q<1$ for all $k\geq 0$. $\square$\\

{\bf Proof of Corollary \ref{corr alg1}}.
\aj{Taking expectation from both {sides} of \zr{\eqref{th2 result},} choosing $N_k=\lceil \rho^{-2k}\rceil$, and using Assumption \ref{assump_error}, one can obtain:}
\begin{align}\label{choose N_k alg1}
\nonumber &\mathbb E\left[ \| x_{T}-\z{\bar x_T }\|\right] \leq\\
\nonumber& (1-q)^{T} \|x_{0}-\z{\bar x_0}\|\\
 & \aj{+ \bar{\alpha} (1-q)^{T-1}\sum_{k=0}^{T-1}\left(\eta\nu (\frac{\rho}{1-q})^k  +\mathbb E[\|e_k\|](1-q)^{-k}\right).}
\end{align}
\z{ Using the fact that $ \sum_{\ze{k}=0}^{T-1} (\z{\tfrac{\rho}{1-q}})^{\ze{k}} = \tfrac{1-(\tfrac{\rho}{1-q})^{T}}{1-\tfrac{\rho}{1-q}} $, in \eqref{choose N_k alg1}, we conclude that }
\begin{align*}
&\mathbb E\left[ \| x_T-\z{\bar x_T} \|\right] \\
& \leq \z{(1-q)}^{T} \|x_{0}-\z{\bar x_0}\|+ \bar{\alpha} \eta \nu \frac{ \rho^{T}-(1-q)^{T}}{\rho+q-1} \\
&\quad +{\bar\alpha (1-q)^{T-1}\sum_{k=0}^{\ze{T-1}}\left(\z{\mathbb E[\|e_k\|]}\z{(1-q)}^{-k}\right).}
\end{align*}
Since $\rho \geq 1-q$ and $q\in (0,1)$, one can easily show that $-\frac{(1-q)^{T}}{\rho+q-1}<0$. Hence, from the previous inequality, we conclude that
\begin{align*}
&\mathbb E\left[ \| x_T-\z{\bar x_T} \|\right] \\
& \leq \z{(1-q)}^{T} \|x_{0}-\z{\bar x_0}\|+ \frac{ \bar{\alpha} \eta \nu \rho^{T}}{\rho+q-1}  \\
&\quad  +{\bar\alpha \sum_{k=0}^{\ze{T-1}}\left(\mathbb E[\|e_k\|]\aj{(1-q)^{T-1-k}}\right)}\\x
&\leq \rho^{T} \|x_{0}-\z{\bar x_0}\|+  \frac{ \bar{\alpha} \eta \nu \rho^{T}}{\rho+q-1} +\bar\alpha \sum_{k=0}^{\ze{T-1}}\left(\mathbb E[\|e_k\|]\rho^{T-1-k}\right).
\end{align*}
\z{According to  {Assumption }\ref{assump:inner}, Algorithm \z{$\mathcal M$} has a convergence rate of $C/t_k^2$ within $t_k$ inner steps. By selecting $t_k=\tfrac{({k+1})\log^2(k+2)}{\rho^{k}}$, we conclude that $\mathbb E[ \|{e_k}\|]\leq \tfrac{C}{t_k}=\tfrac{C\rho^{k}}{(k+1)\log^2(k+2)}$. By using the tower property of expectation in the previous inequality, one can obtain: }
\begin{align*}
&\mathbb E\left[ \| x_T-\z{\bar x_T} \|\right] \leq  \rho^{T} \|x_{0}-\bar x_0\|+ \frac{ \bar{\alpha} \eta \nu \rho^{T}}{\rho+q-1} +\bar\alpha C \sum_{k=0}^{\ze{T-1}} \z{\tfrac{\zr{\rho^{T-1}}}{(k+1)\log^2(k+2)}}.
\end{align*}

Now, by using the fact that $ D\triangleq\sum_{\aj{k=0}}^{\aj{\infty}} \tfrac{1}{(k+1)\log^2(k+2)}\aj{\leq 3.39}$, and rearranging the terms, the desired result can be obtained 
\begin{align*}
   \nonumber&\mathbb E\left[ \| x_T-\z{\bar x_T} \|\right] \leq \z{\rho}^{T} \|x_{0}-\z{\bar x_0}\|+\z{\tfrac{\rho^T \bar\alpha \eta \nu}{\rho+q-1}+\rho^{T-1} \bar \alpha CD.} 
\end{align*}
  \z{{\textbf (ii)}  Similar steps as in the proof of Corollary \ref{corr alg2} (ii). $\square$}

\section{Almost Sure Convergence}\label{sec:a.s.}
{\bf Proof of Theorem \ref{thm:as_convergence}.}
By considering $\bar x_k\in \mathcal{P}_{X^*}(x_k)$ where $X^*$ denotes the set of optimal solutions of problem \eqref{prob:SQVI} and using Lemma \ref{lem:proj-optimal} we conclude that $\bar x_k= \mathcal{P}_{K(\bar x_k)}[\bar x_k-\eta F(\bar x_k)]$. Now using the update rule of $x_{k+1}$, in Algorithm \ref{alg1} and denoting $e_k$ as the error of computing the projection operator, one can obtain:
    \begin{align*}
        \nonumber\|x_{k+1}-\bar x_k\|^2&=\|(1-\alpha_k)x_k+\alpha_k \mathcal{P}_{K(x_k)}\left[x_k-\eta(F(x_k)+\bar w_{k,N_k})\right]+{\alpha_k e_k}- \bar x_k \|^2\\
&= \|x_k-\bar x_k  - \alpha_k (x_k -\mathcal{P}_{K(x_k)}\left[x_k-\eta F(x_k)\right])\\
&\quad +\alpha_k (e_k +\mathcal{P}_{K(x_k)}\left[x_k-\eta(F(x_k)+\bar w_{k,N_k})\right] - \mathcal{P}_{K(x_k)}\left[x_k-\eta F(x_k)\right])\|^2.
    \end{align*}
Defining $ H(x_k) \triangleq \alpha_k (x_k -\mathcal{P}_{K(x_k)}\left[x_k-\eta F(x_k)\right]) $, $e'_k \triangleq e_k +\mathcal{P}_{K(x_k)}\left[x_k-\eta(F(x_k)+\bar w_{k,N_k})\right] - \mathcal{P}_{K(x_k)}\left[x_k-\eta F(x_k)\right] $ and  using $(a - b + c)^2 = a^2 + b^2 + c^2 - 2ab + 2ac -2bc $ we can conclude:
\begin{align}\label{x bound 3p}
   \nonumber \|x_{k+1}-\bar x_k\|^2&= \|x_k - \bar x_k\|^2 +  \|H(x_k)\|^2 +\alpha_k^2\|e'_k\|^2 -2\langle x_{k}-\bar x_k,  H(x_k)\rangle  \\
     &\quad+ 2\langle x_{k}-\bar x_k, \alpha_k e'_k\rangle -2\langle  H(x_k) ,\alpha_k e'_k \rangle.
\end{align}

Furthermore, by using Lemma \ref{lem1} in the definition of $e'_k$ and $(a+b)^2 \leq 2(a^2+b^2)$ one can obtain:
\begin{align}\label{bound e'}
 \nonumber  \|e'_k\|^2 &\leq 2 (\|e_k\|^2 + \|\mathcal{P}_{K(x_k)}\left[x_k-\eta(F(x_k)+\bar w_{k,N_k})\right] - \mathcal{P}_{K(x_k)}\left[x_k-\eta F(x_k)\right]\|^2)\\
    &\leq 2(\|e_k\|^2 + \eta^2\| \bar w_{k,N_k}\|^2).
\end{align}
On the other hand, since $\bar x_k\in\mathcal{P}_{X^*}(x_k)$ is a solution of SQVI, one has: 
\begin{align*}
   \langle \eta F(\bar x_k ), \mathcal{P}_{K(\bar x_k)}\left[x_k-\eta F(x_k)\right] - \bar x_k \rangle  \geq 0.
\end{align*}
Now, let $u = x_k-\eta F(x_k)$ in Lemma \ref{lem1}, since $ \bar x_k  \in X$ then the following holds:
\begin{align*}
    \langle x_k-\eta F(x_k) - \mathcal{P}_{K(\bar x_k)}\left[x_k-\eta F(x_k)\right], \mathcal{P}_{K(\bar x_k)}\left[x_k-\eta F(x_k)\right] -  \bar x_k \rangle \geq 0. 
\end{align*}
Defining  $\widehat H(x_k)\triangleq \alpha_k(x_k-\mathcal{P}_{K(\bar x_k)}\left[x_k-\eta F(x_k)\right])$ and combining this inequality with the previous one, we get:
\begin{align*}
    \langle \tfrac{1}{\alpha_k}\widehat H(x_k)-\eta(F(x_k)-F(\bar x_k)), -\tfrac{1}{\alpha_k}\widehat H(x_k) +(x_k - \bar x_k)\rangle \geq 0.
\end{align*}
Now, rearranging terms and multiplying both sides by $\alpha_k$, one can obtain:
\begin{align}\label{H part bound pre}
\nonumber (x_k - \bar x_k)^T H(x_k) &= (x_k - \bar x_k)^T \widehat H(x_k) +  (x_k - \bar x_k)^T (H(x_k) - \widehat H(x_k))\\
&\geq \tfrac{1}{\alpha_k} \|\widehat H(x_k)\|^2  +\eta \alpha_k (x_k - \bar x_k)^T (F(x_k)-F(\bar x_k))\nonumber\\
&\quad -\eta  (F(x_k)-F(\bar x_k))^T \widehat H(x_k)+(x_k - \bar x_k)^T (H(x_k) - \widehat H(x_k)) \nonumber \\
   &\geq \tfrac{1}{2\alpha_k} \|\widehat H(x_k)\|^2 + \mu_F \eta \alpha_k \|x_k - \bar x_k\|^2 - \tfrac{\eta^2 L^2}{2\tau_k} \|x_k -\bar x_k\|^2 \nonumber\\
   &\quad -\tfrac{\tau_k}{2}\|\widehat H(x_k)\|^2-\alpha_k\gamma \|x_k - \bar x_k\|^2 ,
\end{align}
where the last inequality is obtained by using the QG property, Lipschitz continuity of $F$, implementing Young's inequality, and $|(x_k - \bar x_k)^T (H(x_k) - \widehat H(x_k))|\leq \alpha_k\|x_k-\bar x_k\|\|\mathcal P_{K(x_k)}[u]-\mathcal P_{K(\bar x_k)}[u]\|\leq \alpha_k\gamma \|x_k-\bar x_k\|^2$. 
Now, using the fact that $a^2\geq \frac{1}{1+\xi}b^2-\frac{1+1/\xi}{1+\xi}(a-b)^2$ we obtain $\|\widehat H(x_k)\|^2\geq \frac{1}{1+\xi}\|H(x_k)\|^2-\frac{1}{\xi}\|\widehat H(x_k) - H(x_k)\|^2$ which combined with \eqref{H part bound pre} implies that 
\begin{align}\label{H part bound}
 (x_k - \bar x_k)^T H(x_k)\nonumber&\geq 
\left(\tfrac{1}{\alpha_k}-\tfrac{\tau_k}{2}\right)\left(\tfrac{1}{1+\xi}\right) \| H(x_k)\|^2 -\frac{1}{\xi}\left(\tfrac{1}{\alpha_k}-\tfrac{\tau_k}{2}\right)\|H(x_k)-\widehat H( x_k)\|^2\\&\nonumber\quad+ \left(\mu_F \eta \alpha_k- \tfrac{\eta^2 L^2}{2\tau_k}-\alpha_k\gamma \right) \|x_k - \bar x_k\|^2\nonumber\\
&\geq \left(\tfrac{1}{\alpha_k}-\tfrac{\tau_k}{2}\right)\left(\tfrac{1}{1+\xi}\right) \| H(x_k)\|^2\nonumber\\& + \left(\mu_F \eta \alpha_k- \tfrac{\eta^2 L^2}{2\tau_k}-\alpha_k\gamma - \tfrac{\alpha_k^2\gamma^2}{\xi}\left(\tfrac{1}{\alpha_k}-\tfrac{\tau_k}{2}\right) \right) \|x_k - \bar x_k\|^2.
\end{align}

To make sure $\left(\tfrac{1}{\alpha_k}-\tfrac{\tau_k}{2}\right)\geq0$, we choose $\tau_k=\tfrac{1.8}{\alpha_k}$, then from \eqref{x bound 3p}  we can show that
\begin{align*}
     &\nonumber\|x_{k+1}-\bar x_k\|^2\\&\leq \|x_k - \bar x_k\|^2 +  \|H(x_k)\|^2 +\alpha_k^2\|e_k'\|^2-2\langle x_{k}-\bar x_k,  H(x_k)\rangle \\
     &\quad+ 2\alpha_k\| x_{k}-\bar x_k\|  \|e'_k\| +2\alpha_k\|H(x_k)\|\| e'_k\| \\
     &\leq \|x_k - \bar x_k\|^2 +  \|H(x_k)\|^2 +2\alpha_k^2(\|e_k\|^2 +\eta^2\|\bar w_{k,N_k}\|^2)-2\langle x_{k}-\bar x_k,  H(x_k)\rangle \\
     &\quad+ \alpha_k^2\| x_{k}-\bar x_k\|^2 + 2(\|e_k\|^2 +\eta^2\|\bar w_{k,N_k}\|^2) +\|H(x_k)\|^2+2\alpha_k^2(\|e_k\|^2 +\eta^2\|\bar w_{k,N_k}\|^2) \\
     &\leq \left(1-2 \mu_F \eta \alpha_k + \tfrac{1}{1.8}\eta^2 L^2\alpha_k+\alpha_k^2+2\alpha_k\gamma + \frac{0.2\alpha_k\gamma^2}{\xi}\right)\|x_k - \bar x_k\|^2 -(\tfrac{0.2}{\alpha_k(1+\xi)}-2) \|H(x_k)\|^2 \\ &\quad+(4\alpha_k^2+2)(\|e_k\|^2 +\eta^2\|\bar w_{k,N_k}\|^2),
\end{align*}
where in the first inequality we used Cauchy-Schwartz inequality, the inequality follows from applying Young's inequality and \eqref{bound e'}, and the last one follows from \eqref{H part bound}. 

Next, choosing $\eta= \frac{1.8\mu_F}{L^2}$, $\xi=9$ and $\gamma\leq 45\left(\sqrt{1+0.04\frac{\mu_F^2}{L^2}}-1\right)$ and taking conditional expectation, one can obtain
\begin{align*}
    \mathbb{E}[\|x_{k+1}-\bar x_k\|^2\mid \mathcal F_k]&\leq  (1+\alpha_k^2)\|x_{k}-\bar x_k\|^2 - (\tfrac{0.02}{\alpha_k}-2) \|H(x_k)\|^2 \\  &\quad+2(2\alpha_k^2+1)\big((\tfrac{C\rho^{k}}{(k+1)\log^2(k+2)})^2 +\tfrac{\zz{\eta^2}\nu^2}{N_k}\big), 
\end{align*}
where the last term is obtained by knowing the fact that $\mathbb E[ \|{e_k}\|]\leq \tfrac{C}{t_k}=\tfrac{C\rho^{k}}{(k+1)\log^2(k+2)}$ and Assumption \ref{assump_error}. 
Next, by using the optimality of projection onto a closed convex set and knowing $\bar x_{k+1} = \mathcal{P}_{X^*}(x_{k+1})$, we can bound the left-hand side of the above inequality, we can conclude:
\begin{align*}
    \mathbb{E}[\|x_{k+1}-\bar x_{k+1}\|^2\mid \mathcal F_k]&\leq  (1+\alpha_k^2)\|x_{k}-\bar x_k\|^2 - (\tfrac{0.02}{\alpha_k}-2) \|H(x_k)\|^2 \\ &\quad+2(2\alpha_k^2+1)\big((\tfrac{C\rho^{k}}{(k+1)\log^2(k+2)})^2 +\tfrac{\zz{\eta^2}\nu^2}{N_k}\big). 
\end{align*}
Since $\sum_{k=0}^\infty \tfrac{1}{{N_k}} < \infty$, by choosing $\alpha_k = \tfrac{0.01}{(k+1)^{0.5+\delta}}$ it is easy to show that the requirements of Lemma \ref{robins sig} 
are satisfied, and we can conclude that (i) $\|x_{k}-\bar x_k\|^2$ is a convergent sequence. Hence, the sequence $\{x_k\}_k$ is bounded, which guarantees the existence of a convergent subsequence $x_{k_q} \to x^\#$ as $q \to \infty$, for some $x^\# \in X$. (ii) $\sum_{k=0}^\infty (\tfrac{0.02}{\alpha_k}-2) \|H(x_k)\|^2 < \infty $  almost surely. From the definition of $H(x_k)$, it follows that $\sum_{k=0}^\infty 2\alpha_k(x_k - \mathcal{P}_{K(x_k)}[x_k - \eta F(x_k)]) <\infty$. Since $\sum_{k=0}^\infty\alpha_k=\infty$, we conclude that $\lim_{k \to \infty} x_k - \mathcal{P}_{K(x_k)}[x_k - \eta F(x_k)] \to 0$, which in turn implies that $\lim_{q \to \infty} x_{k_q} - \mathcal{P}_{K(x_{k_q})}[x_{k_q} - \eta F(x_{k_q})] \to 0$. 
 Now considering (i) and (ii) we can show that $ x^\# - \mathcal{P}_{K( x^\#)}[ x^\# - \eta F( x^\#)] = 0$, which means that $ x^\#$ is a solution of \ref{prob:SQVI}. Therefore, we conclude that the sequence $\{x_k\}$ has at least one limit point, and every limit point is a solution of the \ref{prob:SQVI} almost surely. \qed

\end{document}